\documentclass[11pt]{amsart}
\usepackage{amsmath}
\usepackage{amssymb}
\usepackage{amscd}

\def\NZQ{\mathbb}               

\def\QQ{{\NZQ Q}}
\def\ZZ{{\NZQ Z}}

\newtheorem{Theorem}{Theorem}[section]
\newtheorem{Lemma}[Theorem]{Lemma}
\newtheorem{Corollary}[Theorem]{Corollary}
\newtheorem{Proposition}[Theorem]{Proposition}

\newtheorem{Definition}[Theorem]{Definition}

\let\epsilon\varepsilon
\let\phi=\varphi
\let\kappa=\varkappa

\def \s {\sigma}

\begin{document}

\title{ Ramification of valuations and local rings in positive characteristic}
\author{Steven Dale Cutkosky}
\thanks{Partially supported by NSF}

\address{Steven Dale Cutkosky, Department of Mathematics,
University of Missouri, Columbia, MO 65211, USA}
\email{cutkoskys@missouri.edu}


\maketitle

\section{Introduction}
In this paper we consider birational properties of ramification of excellent local rings.

Suppose that $K^*/K$ is a finite separable field extension, $S$ is an excellent local ring of $K^*$ ($S$ has quotient field
$\mbox{QF}(S)=K^*$) and $R$ is an excellent  local ring of $K$ such that $\dim S=\dim R$, $S$  dominates $R$ ($R\subset S$ and the maximal ideals $m_S$ of $S$ and $m_R$ of $R$ satisfy  $m_S\cap R=m_R$) and $\nu^*$ is a valuation of $K^*$ which dominates $S$ (the valuation ring $V_{\nu^*}$ of $\nu^*$ dominates $S$).  Let $\nu$ be the restriction of $\nu^*$ to $K$.

The notation that we use in this paper  is explained  in more detail in Section \ref{SecNoc}.

\subsection{Local Monomialization.}


\begin{Definition}\label{DefMon} $R\rightarrow S$ is {\it monomial} if 
$R$ and $S$ are regular local rings of the same dimension $n$ and  there exist regular systems of parameters $x_1,\ldots,x_n$ in $R$ and $y_1,\ldots,y_n$ in $S$, units $\delta_1,\ldots,\delta_n$ in $S$
and an $n\times n$ matrix $A=(a_{ij})$ of natural numbers with nonzero determinant such that
\begin{equation}\label{eqI1}
x_i=\delta_i\prod_{j=1}^ny_j^{a_{ij}}\mbox{ for $1\le i\le n$.}
\end{equation}
\end{Definition}

If $R$ and $S$ have equicharacteristic zero and algebraically closed residue fields, then within the extension $\hat R\rightarrow \hat S$ there are regular parameters giving a form (\ref{eqI1}) with all $\delta_i=1$.

More generally, we ask if a given extension $R\rightarrow S$ has a local  monomialization along the valuation $\nu^*$.

\begin{Definition}\label{DefI1} A {\it local monomialization} of $R\rightarrow S$ is a commutative diagram
$$
\begin{array}{lllll}
R_1&\rightarrow &S_1&\subset& V_{\nu^*}\\
\uparrow&&\uparrow&&\\
R&\rightarrow&S&&
\end{array}
$$
such that the vertical arrows are products of monoidal transforms (local rings of blowups of regular primes)
and $R_1\rightarrow S_1$ is monomial.
\end{Definition}

It is proven in Theorem 1.1 \cite{C} that a local monomialization always exists when $K^*/K$ are algebraic function fields over a (not necessarily algebraically closed)
field $k$ of characteristic 0, and $R\rightarrow S$ are algebraic local rings of $K$ and $K^*$ respectively. (An algebraic local ring  is  essentially of finite type over $k$.)

We can also define the weaker notion of a {\it weak local monomialization} by only requiring that the conclusions of Definition
\ref{DefI1} hold with the vertical arrows  being required to be birational (and not necessarily factorizable by products of monoidal transforms).

This leads to the following question for extensions $R\rightarrow S$ as defined in the beginning of this paper:
{\it When does there  exist a local monomialization (or at least a weak local monomialization) of an extension $R\rightarrow S$ of excellent local rings dominated by a valuation $\nu^*$ ?}

As commented above, the question has a positive answer within  algebraic function fields over an arbitrary field of characteristic zero by Theorem 1.1 \cite{C}.

For the question to have a positive answer in positive characteristic or mixed characteristic it is of course necessary that some form of resolution of singularities be true.
This is certainly true in equicharacteristic zero, and is 
known to be true very generally in dimension $\le 2$ (\cite{LU}, \cite{Li}, \cite{CJS}) and in positive characteristic and dimension 3 (\cite{RES}, \cite{C1}, \cite{CP1} and \cite{CP2}). A few recent papers going beyond dimension three are \cite{dJ}, \cite{Ha}, \cite{BrV}, \cite{BeV}, \cite{H1}, \cite{T1}, \cite{T2}, \cite{KK} and \cite{Te}.

The case of two dimensional algebraic function fields over an algebraically closed field of positive characteristic is considered in \cite{CP}, where it is shown that monomialization is true  if $R\rightarrow S$ is a {\it defectless} extension  of two dimensional algebraic local rings over an algebraically closed field $k$ of  characteristic $p>0$ (Theorem 7.3 and Theorem 7.35 \cite{CP}).  We will discuss the important concept of  defect later on in this introduction.

In \cite{C2},  we give an example showing  that weak monomialization (and hence monomialization) does not exist in general for extensions of  algebraic local rings of dimension $\ge 2$ over a field  $k$ of  char $p>0$. We prove the following theorem in \cite{C2}:

\begin{Theorem}\label{TheoremNM2} (Theorem 1.4 \cite{C2}, Counterexample to local and weak local monomialization) Let $k$ be a field of characteristic $p>0$ with at least 3 elements and let $n\ge 2$. Then
there exists a finite separable extension $K^*/K$ of  n dimensional  function fields over $k$,  a valuation $\nu^*$ of $K^*$ with restriction $\nu$ to $K$ and algebraic regular local rings $A$ and $B$ of $K$ and $K^*$ respectively, such that $B$ dominates $A$, $\nu^*$ dominates $B$ and there do not exist
 regular algebraic local rings $A'$ of $K$ and $B'$ of $K^*$ such that $\nu^*$ dominates $B'$, $B'$ dominates $A'$, $A'$ dominates $A$, $B'$ dominates $B$ and $A'\rightarrow B'$ is monomial.
\end{Theorem}

We have that the defect $\delta(\nu^*/\nu)=2$ in the example of Theorem \ref{TheoremNM2} (with $\nu=\nu^*|K$).

In \cite{C} and \cite{CP}, a very strong form of local monomialization is established within characteristic zero algebraic function fields which we call strong local monomialization (Theorem 5.1 \cite{C} and Theorem 48 \cite{CP}).
This form is stable under appropriate sequences of monoidal transforms and encodes the classical invariants of the extension of valuation rings. In \cite{CP}, we show that strong local monomialization is true for defectless extensions of two dimensional algebraic function fields
(Theorem 7.3 and Theorem 7.35 \cite{CP}).
We give an example in \cite{CP} (Theorem 7.38 \cite{CP}) showing that strong local monomialization is not generally true for defect extensions of two dimensional algebraic function  fields (over a field of positive characteristic).

In this paper, we establish that local monomialization (and strong local monomialization) hold for defectless extensions of two dimensional excellent local rings. We will first state our theorem (which is proven in Section \ref{SecSM}), and then we will define and discuss the defect.

\begin{Theorem}\label{Theorem7} Suppose that $R$ is a 2 dimensional excellent local domain with quotient field $\mbox{QF}(R)=K$. Further suppose that $K^*$ is a finite separable extension of $K$ and $S$ is a two dimensional excellent local domain with quotient field
$\mbox{QF}(S)=K^*$ such that $S$ dominates $R$. Let $\nu^*$ be a valuation of $K^*$ which dominates $S$ and let $\nu$ be the restriction of $\nu^*$ to $K$. Suppose that the defect $\delta(\nu^*/\nu)=0$. Then 
there exists a commutative diagram  
\begin{equation}
\begin{array}{ccccc}
R_1&\rightarrow& S_1&\subset & V_{\nu^*}\\
\uparrow&&\uparrow&&\\
R&\rightarrow &S&&
\end{array}
\end{equation}
such that the vertical arrows are products of quadratic transforms along $\nu^*$ and $R_1 \rightarrow S_1$ is  monomial.
\end{Theorem}

The proof of the theorem actually produces stable strong monomialization.

We now define the defect of an extension of valuations. The role of this concept in local uniformization was  observed by Kuhlmann \cite{Ku1} and \cite{Ku2}.
A good introduction to the role of defect in valuation theory is given in \cite{Ku1}.
A brief survey which is well suited to our purposes is given in Section 7.1 of \cite{CP}.
Suppose that $K^*/K$ is a finite Galois extension of fields of characteristic $p>0$. The splitting field  $K^s(\nu^*/\nu)$ of $\nu$ is the smallest field between $K$ and   $K^*$ with the property  that $\nu^*$ is the only extension to $K^*$ of $\nu^*|L$. The defect $\delta(\nu^*/\nu)$ is defined by the identity
$$
[K^*:K^s(\nu^*/\nu)]=f(\nu^*/\nu)e(\nu^*/\nu)p^{\delta(\nu^*/\nu)}
$$
(Corollary to Theorem 25 , Section 12, Chapter VI \cite{ZS2}).
In the case when $K^*/K$ is only finite separable, we define the defect by
$$
\delta(\nu^*/\nu)=\delta(\nu'/\nu)-\delta(\nu'/\nu^*)
$$
where $\nu'$ is an extension of $\nu^*$ to a Galois closure $K'$ of $K^*$ over $K$.

The defect is equal to zero if the residue field $V_{\nu}/m_{\nu}$  has characteristic zero (Theorem 24, Section 12, Chapter VI \cite{ZS2}) or if $V_{\nu}$ is a DVR (Corollary to Theorem 21, Section 9, Chapter V \cite{ZS1}).

\subsection{Associated graded rings of  valuations.}

The semigroup of $R$ with respect to the valuation $\nu$ is 
$$
S^R(\nu)=\{\nu(f)\mid f\in R\setminus \{0\}\}.
$$
The group generated by $S^R(\nu)$ is the valuation group $\Gamma_{\nu}$ of $\nu$ which is well understood (\cite{M}, \cite{MS}, \cite{ZS2}, \cite{Ku2}); the semigroup 
can however be extremely complicated and perverse (\cite{CT1}, \cite{CV1}). 

The associated graded ring of $\nu$ on $R$, as defined in \cite{T1} and \cite{T2},  is 
$$
{\rm gr}_{\nu}(R)=\bigoplus_{\gamma\in S^R(\nu)}\mathcal P_{\gamma}(R)/\mathcal P^+_{\gamma}(R).
$$
Here $P_{\gamma}(R)$ is the ideal in $R$ of elements of value $\ge \gamma$ and $P_{\gamma}^+(R)$ is the ideal in $R$ of elements of value $>\gamma$. This ring plays an important role in Teissier's approach to resolution of singularity (it is completely realized for Abhyankar valuations in arbitrary  characteristic in \cite{T2}).

We always have that $\mbox{QF}({\rm gr}_{\nu}(R))=\mbox{QF}({\rm gr}_{\nu}(V_{\nu})$ and
$$
[\mbox{QF}({\rm gr}_{\nu^*}(S)):\mbox{QF}({\rm gr}_{\nu}(R))]=f(\nu^*/\nu)e(\nu^*/\nu).
$$

In \cite{GHK} and \cite{GK} it is proven that there exists a strong local monomialization $R_1\rightarrow S_1$ for defectless extensions $R\rightarrow S$ of two dimensional algebraic local rings in a two  algebraic function field over an algebraically closed field, which has the property that the induced extension of associated graded rings along the valuation 
\begin{equation}\label{eqI2}
{\rm gr}_{\nu}(R)\rightarrow {\rm gr}_{\nu^*}(S)
\end{equation}
is of finite type and is even a ``toric extension''. This result is extended in \cite{CV2} to the case of two dimensional algebraic function fields over an arbitrary field of characteristic zero. These proofs all make use of the technique of generating sequences of a valuation on a local ring, which is developed by Spivakovsky  in \cite{Sp} for two dimensional regular local rings with algebraically closed residue fields, and is extended in \cite{CV1} to arbitrary regular local rings of dimension two. Unfortunately, this technique is special to dimension two, and does not extend well to higher dimension local rings, or even to normal local rings of dimension two
(the examples of strange semigroups in \cite{CT1} and \cite{CV1}  show this). An interesting  construction of generating sequences within a valuation ring 
which exhibits the defect of an extension of valuations is given in \cite{Vaq}, and a different general construction of generating sequences is given in \cite{Mo}. 

In general, the extension (\ref{eqI2}) is not of finite type, even for equicharacteristic zero algebraic regular local rings of dimension two (Example 9.4 \cite{CV1}), so blowing up to reach a good stable form is required to obtain that (\ref{eqI2}) has a good form.

It is not difficult to show that the extension (\ref{eqI2}) is of finite type and is toric when $\nu^*$ is an Abhyankar valuation
(equality holds in Abhyankar's inequality (Theorem 1 \cite{Ab1})
$$
\mbox{trdeg}_{S/m_S}V_{\nu^*}/m_{\nu^*}+\dim_{\QQ}\Gamma_{\nu^*}\otimes \QQ\le \dim S.
$$

From a special case of  Theorem 5.1 \cite{C} (recalled in Theorem \ref{Theorem1} of this paper) we give the precise statement of the stable strongly monomial forms 
$R_1\rightarrow S_1$ obtained by a rational rank 1 valuation ($\dim_{\QQ}\Gamma_{\nu^*}\otimes \QQ = 1$) in the case when $R\rightarrow S$ is an extension of  algebraic local rings in an extension of algebraic function fields over an arbitrary field of  characteristic zero. In this case, there are regular parameters $x_1,\ldots, x_n$ in $R_1$ and regular parameters $y_1,\ldots,y_n$ in $S$ and a unit $\delta$ in $S_1$ such that 
$$
x_1=\delta y_1^{e}, x_2=y_2,\ldots,x_n=y_n
$$
where $e=e(\nu^*/\nu)=|\Gamma_{\nu^*}/\Gamma_{\nu}|$.
In this paper, we give a simple proof (Theorem \ref{Theorem4}) that  in the case when $R\rightarrow S$ is an extension of  algebraic local rings in an extension of algebraic function fields over an arbitrary field of  characteristic zero, and $\nu$ has rational rank 1  a strongly monomial extension $R_1\rightarrow S_1$ has the property that 
\begin{equation}\label{eqI4}
{\rm gr}_{\nu^*}(S_1)\cong \left({\rm gr}_{\nu}(R_1)\otimes_{R_1/m_{R_1}}S_1/m_{S_1}\right)[Z]/(Z^e-[\delta]^{-1}[x_1]),
\end{equation}
where $[\gamma_1], [x_1]$ are the respective classes in ${\rm gr}_{\nu}(R_1)\otimes_{R_1/m_{R_1}}S_1/m_{S_1}$.
The degree of the extension of quotient fields of ${\rm gr}_{\nu}(R)\rightarrow {\rm gr}_{\nu^*}(S)$ is $e(\nu^*/\nu)f(\nu^*/\nu)$,
where 
$$
f(\nu^*/\nu)=[V_{\nu^*}/m_{\nu^*}:V_{\nu}/m_{\nu}].
$$
In particular, the extension of associated graded rings along the valuation is finite and ``toric''.

We show in Theorem \ref{Theorem3} of this paper that the stable strongly monomial forms  found in Theorem \ref{Theorem7} of defectless extensions of two dimensional excellent local rings dominated by a valuation $\nu^*$ of  rational rank 1
have an extension of associated graded rings along the valuation of the form (\ref{eqI4}). Since stable forms of  Abhyankar valuations have a finite type ``toric extension'' as commented above, we conclude that stable strongly monomial forms of 
a defectless extension of two dimensional excellent local rings always has a finite type ``toric'' extension (\ref{eqI2}).

In contrast, we do not have such a nice stable form of the extension of associated graded rings along a valuation which has positive defect,
as is shown by the example of Theorem 38 \cite{CP}, analyzed in Section \ref{SectionCSM} of this paper. Using the notation of Section 7.4 \cite{CP}, explained in the following subsection on invariants of stable forms along a valuation,
it follows from Remark 7.44 \cite{CP}  that the graded domains $\mbox{gr}_{\nu^*}(S_{n})$ are integral but not finite
over $\mbox{gr}_{\nu}(R_{n})$ for all  $n$, in contrast to the situation when the defect $\delta(\nu^*/\nu)$ is zero in Theorem \ref{Theorem3}.
The quotient fields of 
$\mbox{gr}_{\nu}(R_{n})$ and $\mbox{gr}_{\nu^*}(S_{n})$  are equal under the extension of Theorem 38 \cite{CP}, so the degree is  $1=e(\nu^*/\nu)f(\nu^*/\nu)$
as in the conclusions of Theorem \ref{Theorem3}.

\subsection{Invariants of stable forms along a valuation.}

Suppose that  $R\rightarrow S$ is an inclusion of regular two dimensional algebraic local rings  within function fields $K$ and $K^*$ respectively, over an algebraically closed field $k$ of positive characteristic $p>0$, such that $K^*/K$ is finite separable, $S$ 
dominates $R$, and  there is a valuation $\nu^*$ of $K^*$ with restriction $\nu$ to $K$ which dominates $R$ and $S$, such that 
\begin{enumerate}
\item[1)] $\nu^*$ dominates $S$.
\item[2)] The residue field $V_{\nu^*}/m_{\nu^*}$ of $V_{\nu^*}$ is algebraic over $S/m_S$. 
\item[3)] The value group $\Gamma_{\nu^*}$ of $\nu^*$ has rational rank 1 (so it is isomorphic as an ordered group to a subgroup of $\QQ$).
\end{enumerate}

\vskip .2truein

It is shown in Corollary 7.30 and Theorem 7.33 \cite{CP} 
that there  are sequences of quadratic transforms along $\nu^*$ (each vertical arrow is a product of quadratic transforms) constructed by the algorithm of Section 7.4 \cite{CP}, where we have simplified notation, writing $R_n$ for $R_{r_n}$ and $S_n$ for $S_{s_n}$.

\begin{equation}\label{eq100}
\begin{array}{ccc}
V_{\nu}&\rightarrow &V_{\nu^*}\\
\uparrow&&\uparrow\\
\vdots&&\vdots\\
\uparrow&&\uparrow\\
R_{n}&\rightarrow &S_{n}\\
\uparrow&&\uparrow\\
\vdots&&\vdots\\
\uparrow&&\uparrow\\
R_{2}&\rightarrow&S_{2}\\
\uparrow&&\uparrow\\
R&\rightarrow&S
\end{array}
\end{equation}
where  each $R_{n}$ is an algebraic regular local ring of $K$ and $S_{n}$ is an algebraic regular local ring of $K^*$ such that $S_{n}$ dominates $R_{n}$,
and no quadratic transform of $R_{n}$ factors through $S_{n}$.  For $n\gg0$,
$R_{n}$ has regular parameters $x_{n}, y_{n}$ and $S_{n}$ has regular parameters $u_{n}, v_{n}$ such that there are ``stable forms''
\begin{equation}\label{eqI5}
u_{n}=\gamma_nx_{n}^{\overline a p^{\alpha_n}}, v_{n}=x_{n}^{b_n}f_n
\end{equation}
where $\gamma_n$ is a unit in $S_{n}$, $d_n=\overline \nu_{n}(f_n\mbox{ mod }x_{n})=p^{\beta_n}$, where $\overline \nu_{n}$ is the natural valuation of the DVR $S_{n}/x_{n}S_{n}$, with $b_n,\alpha_n,\beta_n\ge 0$ and $\alpha_n+\beta_n$ does not depend on $n$, $\overline a$ does not depend on $n$.

In Theorem 7.38 \cite{CP}, an example is given where $K^*/K$ is a tower of two Artin Schreier extensions with $\delta(\nu^*/\nu)=2$,
 $\overline a=1$, $\alpha_n=1$ for all $n$, $\beta_n=1$ for all $n$ and $b_n=0$ for all $n$. In particular, this shows that ``strong local monomialization'' fails for this extension.
However, it is also shown in the example that local monomialization is true for this extension (by considering different sequences of quadratic transforms above $R$ and $S$).

In Corollary 7.30 and Theorem 7.33 of \cite{CP}, it is shown that $\alpha_n+\beta_n$ is a constant for $n\gg 0$, where  $\alpha_n$ and $\beta_n$ are the integers defined above which are associated to the stable forms (\ref{eq100}) of an extension of valued two dimensional algebraic function fields.  If $\Gamma_{\nu}$ is not $p$-divisible, it is 
further shown that $\alpha_n$ and $\beta_n$ are both constant for $n\gg 0$, and $p^{\beta_n}$ is the defect of the extension.
However, if $\Gamma_{\nu}$ is $p$-divisible, then it is only shown that the sum $\alpha_n+\beta_n$ is constant for $n\gg 0$, and that 
$p^{\alpha_n+\beta_n}$ is the defect of the extension. In Remark 7.34 \cite{CP} it is asked if $\alpha_n$ and $\beta_n$ (and some other numbers computed from  the stable forms) are eventually constant in the case when $\Gamma_{\nu}$ is $p$-divisible. We give examples in Section \ref{SectionCSM}, equations (\ref{eq500}) - (\ref{eq503}),  showing that this is not the case, even within (defect) Artin Schreier extensions. The examples are found by considering a factorization of the example of Theorem 3.8 \cite{CP} into a product of two Artin Schreier extensions, and computing generating sequences on the intermediary rings.

\section{Notation and Preliminaries}\label{SecNoc}
\subsection{Local algebra.}  All rings will be commutative with identity. A ring $S$ is essentially of finite type over $R$ if $S$ is a local ring of a finitely generated $R$-algebra.
We will denote the maximal ideal of a local ring  $R$ by $m_R$, and the quotient field of a domain $R$ by $\mbox{QF}(R)$. 
(We do not require that a local ring be Noetherian). 
Suppose that $R\subset S$ is an inclusion of local rings. We will say that $S$ dominates $R$ if $m_S\cap R=m_R$.
If the local ring $R$ is a domain with $\mbox{QF}(R)=K$ then we will say that $R$ is a local ring of $K$. If $K$ is an algebraic function field over a field $k$ (which we do not assume to be algebraically closed) and a local ring $R$ of $K$ is essentially of finite type over $k$, then we say that $R$ is an algebraic local ring of $k$.  

Suppose that $K\rightarrow K^*$ is a finite field extension, $R$ is a local ring of $K$ and $S$ is a local ring of $K^*$. We will say that $S$ lies over $R$ if $S$ is a localization of the integral closure $T$ of $R$ in $K^*$. If $R$ is a local ring, $\hat R$ will denote the completion of $R$ by its maximal ideal $m_R$.

Suppose that $R$ is a regular local ring. A monoidal transform $R\rightarrow R_1$ of $R$ is a local ring of the form $R[\frac{P}{x}]_m$
where $P$ is a regular prime ideal in $R$ ($R/P$ is a regular local ring) and $m$ is a prime ideal of $R[\frac{P}{x}]$ such that
$m\cap R=m_R$. $R_1$ is called a quadratic transform if $P=m_R$.

\subsection{Valuation Theory}
Suppose that $\nu$ is a valuation on a field $K$. We will denote by $V_{\nu}$ the valuation ring of $\nu$:
$$
V_{\nu}=\{f\in K\mid \nu(f)\ge 0\}.
$$
We will denote the value group of $\nu$ by $\Gamma_{\nu}$. Good treatments of valuation theory are Chapter VI of \cite{ZS2} and \cite{RTM}, which contain references to the original papers.
If $\nu$ is a valuation ring of an algebraic function field over a field $k$, we  insist that $\nu$ vanishes on $k\setminus\{0\}$,
and say that $\nu$ is a $k$ valuation.

If $\nu$ is a valuation of a field $K$ and $R$ is a local ring of $K$ we will say that $\nu$ dominates $R$ if the valuation ring
$V_{\nu}$ dominates $R$. Suppose that $\nu$ dominates $R$. A monoidal transform $R\rightarrow R_1$ is called a monoidal transform along $\nu$ if $\nu$ dominates $R_1$.

Suppose that $K^*/K$ is a finite separable extension, $\nu^*$ is a valuation of $K^*$ and $\nu$ is the restriction of $\nu$ to $K$.
The ramification index is
$$
e(\nu^*/\nu)=|\Gamma_{\nu^*}/\Gamma_{\nu}|
$$
and reduced degree is
$$
f(\nu^*/\nu)=[V_{\nu^*}/m_{\nu^*}:V_{\nu}/m_{\nu}].
$$
The defect $\delta(\nu^*/\nu)$ is defined in the introduction to this paper. Its basic properties are developed  in  Section 11, Chapter VI \cite{ZS2},
\cite{Ku1} and
Section 7.1 of \cite{CP}. 

We will call a ring a DVR if it is a valuation ring with value group $\ZZ$.

\subsection{Galois theory of local rings} Suppose that $K^*/K$ is a finite Galois extension, $R$ is a local ring of $K$ and $S$ is a local ring of $K^*$ which lies over $R$. The splitting group $G^s(S/R)$, splitting field $K^s(S/R)=(K^*)^{G^s(S/R)}$  and inertia group $G^i(S/R)$   are defined and their basic properties developed in Section 7 of \cite{RTM}. 

\subsection{Galois theory of valuations}
The Galois theory of valuation rings is developed in Section 12 of Chapter VI of \cite{ZS2} and in Section 7 of \cite{RTM}.
Some of the basic results we need are surveyed in Section 7.1 \cite{CP}.
If we take $S=V_{\nu^*}$ and 
$R=V_{\nu}$ where $\nu^*$ is a valuation of $K^*$ and $\nu$ is the restriction of $\nu$ to $K$, then we obtain the splitting group
$G^s(\nu^*/\nu)$, the splitting field $K^s(\nu^*/\nu)$ and the inertia group $G^i(\nu^*/\nu)$. 
In Section 12 of Chapter VI of \cite{ZS2}, $G^s(\nu^*/\nu)$ is written as $G_Z$ and called the decomposition group.
$G^i(\nu^*/\nu)$ is written as $G_T$. The ramification group $G_V$ of $\nu^*/\nu$ is  defined in Section 12 of Chapter VI of \cite{ZS2} and is surveyed in Section 7.1 \cite{CP}. We  will denote this group by $G^r(\nu^*/\nu)$.  
\subsection{Semigroups and associated graded rings of a local ring with respect to a valuation}
Suppose that $\nu$ is a valuation of field $K$ which dominates a local ring $R$ of $K$. We will denote the semigroup of values of $\nu$ on $S$ by
$$
S^R(\nu)=\{\nu(f)\mid f\in R\setminus \{0\}\}.
$$
Suppose that $\gamma\in \Gamma_{\nu}$. We define ideals in $R$
$$
\mathcal P_{\gamma}(R)=\{f\in R\mid \nu(f)\ge 0\}
$$
and
$$
\mathcal P_{\gamma}^+(R)=\{f\in R\mid \nu(f)> 0\}
$$
and define (as in \cite{T1}) the associated graded ring of $R$ with respect to $\nu$ by
$$
{\rm gr}_{\nu}(R)=\bigoplus_{\gamma\in S^R(\nu)}\mathcal P_{\gamma}(R)/\mathcal P_{\gamma}^+.
$$
\subsection{Birational geometry of two dimensional regular local rings}
We recall some basic theorems  which we will make frequent use of.

\begin{Theorem}\label{ThmPre1}(Theorem  3 \cite{Ab1}) Suppose that $K$ is a field, and $R$ is a regular local ring of dimension two of $K$. Suppose that $S$ is another 2 dimensional regular local ring of $K$ which dominates $R$. Then there exists a unique sequence of quadratic transforms of $R$ 
$$
R\rightarrow R_1\rightarrow \cdots \rightarrow R_n=S
$$
of $R$ which factor $R\rightarrow S$.
\end{Theorem}

\begin{Lemma}\label{LemmaPre2}(Lemma 12 \cite{Ab1}) Suppose that $A$ is a two dimensional regular local ring of a field $K$ and $\nu$ is a valuation of $K$ which dominates $\nu$.
Let 
$$
R\rightarrow R_1\rightarrow R_2\rightarrow \cdots
$$
be the infinite sequence of quadratic transforms along $\nu$. 
Then
$$
V_{\nu^*}=\cup_{i=1}^{\infty}R_i.
$$
\end{Lemma}

We also make use of the fact that ``embedded resolution of singularities '' is true within a regular local ring of dimension 2 (Theorem 2 \cite{Ab1}),
and the fact that resolution of singularities is true for two dimensional excellent local rings (\cite{Li}, \cite{CJS}).

\section{Local monomialization of  two dimensional defectless extensions}\label{SecSM}

In this section we prove Theorem \ref{Theorem7}, establishing local monomialization for defectless extensions of two dimensional excellent local domains.
This extends the result for extensions of two dimensional algebraic local rings in two dimensional algebraic function fields over an algebraically closed field in Theorem 7.3 and 7.35 \cite{CP}.

We indicate the differences between the proof in \cite{CP} of the analogue of Theorem \ref{Theorem7} for algebraic local rings over an algebraically closed field, and the proof of Theorem \ref{Theorem7} in this paper. The essential case is of rational rank 1 valuations (Theorem \ref{Theorem2}).
An essential ingredient in the proof is the computation of complexity in Proposition \ref{Prop10}, which generalizes Proposition 7.2 \cite{CP}. The steps of the proof are the same, but some of the individual calculations  require different methods, as we do not have coefficient fields in general in the situation of this paper, and the completions of local rings are no longer extensions of power series rings over a field. In \cite{CP}, the analogue of Theorem \ref{Theorem2} is deduced as a consequence of a detailed analysis of stable forms (Theorem 7.33 \cite{CP}) which makes essential use of the assumption that there is no residue field extension (the ground field is algebraically closed). In this paper, we give a different, more direct argument to deduce Theorem \ref{Theorem2}.

The proof of  Theorem \ref{Theorem7} actually produces stable strong monomialization. We first establish strong monomialization in the two 
essential cases of the theorem, and give the proof of Theorem \ref{Theorem7} at the end of this section.
We will make use of the list of good properties of excellent rings given in Scholie 7.8.3 \cite{EGAIV}.

\subsection{Degree formulas} In this subsection we generalize the formulas of Proposition 7.2 \cite{CP}.

\begin{Proposition}\label{Prop10}   Suppose that $R$ and $S$ are two dimensional regular excellent local rings such that $S$ dominates $R$ and 
$K^*=\mbox{QF}(S)$ is a finite separable extension of $K=\mbox{QF}(R)$, $R$ has a regular system of parameters $u,v$ and $S$ has a regular system of parameters $x,y$ such that there is an expression
$$
u=\gamma x^a, v=x^bf
$$
where $a>0$, $b\ge 0$, $\gamma$ is a unit in $S$, $x\not\,\mid f$ in $S$ and $f$ is not a unit in $S$.
Then there exist inclusions
$$
R\rightarrow R_0\rightarrow S
$$
of local rings  with the following properties: $R_0$ is a two dimensional  normal local ring of $K$ which is essentially of finite type over $R$  such that $S$ lies above $R_0$. Let $\overline \nu$ be the natural valuation of the DVR $S/xS$. Then we have 
$$
[\mbox{QF}(\hat S):\mbox{QF}(\hat R_0)]=ad[S/m_S:R/m_R]
$$
where $d=\overline\nu(f\mbox{ mod } x)$.
\end{Proposition}

\begin{proof} First suppose that $b>0$. Let $s=\mbox{gcd}(a,b)$, $I$ be the ideal which is the integral closure of 
$(u^{\frac{b}{s}}, v^{\frac{a}{s}})$ in $R$. Let 
$$
\phi=\frac{v^{\frac{a}{s}}}{u^{\frac{b}{s}}}
$$
and
$$
R_0=R\left[\frac{1}{u^{\frac{b}{s}}}I\right]_{m_S\cap R\left[\frac{1}{u^{\frac{b}{s}}}I\right]}.
$$
$R_0$ is normal since $I$ is integrally closed. (Powers of $I$ are integrally closed by Theorem 2', Appendix 5 \cite{ZS2}, so the ring $\bigoplus_{n\ge 0}I^n$ is integrally closed and $R_0$ is a local ring of the normal scheme $\mbox{Proj}(\bigoplus_{n\ge 0}I^n)$). 
The elements
$x^a,\phi\in m_{R_0}S$ so $m_{R_0}S$ is $m_S$-primary.

By Lemma 9 \cite{Ab1}, there exists a rank 2 valuation $\nu$ of $K^*$ which dominates $S$. Let $A$ be the integral closure of $R_0$ in $K^*$ and let $T=A_{m_S\cap A}$.  By Theorem 1 \cite{Ab1}, $V_{\nu}/m_{\nu}$ is finite over $T/m_T$ so $S/m_S$ is finite over $T/m_T$. $T$ is normal of dimension two (since $T$ is excellent and normal). Since $m_TS$ is $m_S$-primary, we have that $T=S$ by the version of Zariski's main theorem in (10.7) \cite{RES}. 
Thus $S$ lies over $R_0$.

Let $W^*$ be the DVR $W^*= S_{(x)}$. $\phi\in R_0$ is neither a unit nor divisable by $x$ in $S$. Thus the prime ideal $\mathfrak p=xS\cap R_0$ has height one in $R_0$. Thus 
 $$
W=(R_0)_{\mathfrak p}=W^*\cap K
$$
is a  DVR. 

Let $t$ be a regular parameter in $W$. $W$ is the valuation ring of the valuation $\mbox{ord}_t$. Since $I$ is a monomial ideal,
the value group $\ZZ$ of $\mbox{ord}_t$ is generated by $\mbox{ord}_tu$ and $\mbox{ord}_tv$. Thus $\mbox{gcd}(\mbox{ord}_tu,\mbox{ord}_tv)=1$.
Since $\mbox{ord}_t\phi=0$, we have that $\mbox{ord}_tu=\frac{a}{s}$ and $\mbox{ord}_tv=\frac{b}{s}$. 

We have that 
$$
IW=u^{\frac{b}{s}}W=(t^{\frac{a}{s}\frac{b}{s}}).
$$
Since 
$$
IW^*=(x^{\frac{ab}{s}}),
$$
we have that  $tW^*=(x^s)$, so the ramification index of $W^*/W$ is
$$
e(W^*/W)=s.
$$
The integrally closed ideal $I$ is generated by all monomials $u^mv^n$ such that 
$$
\frac{m}{b/s}+\frac{n}{a/s}\ge 1.
$$
Since
$$
\mbox{gcd}(\frac{a}{s},\frac{b}{s})=1,
$$
we have that 
\begin{equation}\label{eq521}
\mbox{ord}_x(u^mv^n)=ma+nb>\frac{ab}{s}
\end{equation}
for  such a monomial provided 
$$
(m,n)\not\in \{(\frac{b}{s},0),(0,\frac{a}{s})\}.
$$
Thus for $u^mv^n\in I$,
$$
\frac{u^mv^n}{u^{\frac{b}{s}}}\in \mathfrak p
$$
unless 
$$
u^mv^n=u^{\frac{b}{s}}\mbox{ or }u^mv^n=v^{\frac{a}{s}}.
$$
Since $\overline \nu(\phi)=\frac{ad}{s}>0$, we have that 
\begin{equation}\label{eq522*}
R_0/\mathfrak p\cong (R/m_R[\overline\phi])_{(\overline\phi)}
\end{equation}
where $\overline \phi$ is the residue of $\phi$ in $R_0/\mathfrak p$, and
\begin{equation}\label{eq540}
R_0/m_{R_0}=R/m_R.
\end{equation}

By Corollary 1 of Section 2, Chapter II  of Local Fields \cite{S} and (ii) of Theorem 1, Section 3 of Chapter II \cite{S}, and by (\ref{eq540}) and  the fact that $\overline\phi$ is a regular parameter in $\widehat{R_0/\mathfrak p}$ by (\ref{eq522*}), 
we have that the inclusion
$$
\widehat{R_0/\mathfrak p}\rightarrow \widehat{S/xS}
$$
is a finite extension of complete DVRs, and
$$
\begin{array}{lll}
[\mbox{QF}(\widehat{S/xS}):\mbox{QF}(\widehat{R_0/\mathfrak p})]&=&e(\widehat{S/xS}/\widehat{R_0/\mathfrak p})f(\widehat{S/xS}/\widehat{R_0/\mathfrak p})\\
&=&e(S/xS/R_0/\mathfrak p)f(S/xS/R_0/\mathfrak p)=\frac{ad}{s}[S/m_S:R/m_R].\end{array}
$$

We have that $m_RS\subset \mathfrak p S$ so $xS$ is the only prime ideal of $S$ lying over $\mathfrak p$. Further, $\mathfrak p\hat R_0$ and $x\hat S$ are prime ideals since $R_0/\mathfrak p$ and $S/xS$ are regular local rings. $\hat R_0$ and $\hat S$ are normal since $R_0$ and $S$ are normal and excellent. Also,  $\hat S$ is a finite extension of $\hat R_0$ by (10.13) and (10.2) of \cite{RES}, since $S$ lies over  $R_0$ . Thus
$\mbox{QF}(\hat S)$ is a finite field extension of $\mbox{QF}(\hat R_0)$ and $\hat S$ is the integral closure of $\hat R_0$ in $\mbox{QF}(\hat S)$. We obtain that $\hat W^*=\hat S_{x\hat S}$ is the unique DVR of $\mbox{QF}(\hat S)$ which dominates $\hat W=(\hat R_0)_{\mathfrak p \hat R_0}$.

Thus, by Theorem 20, page 60 \cite{ZS2},
$$
[\mbox{QF}(\hat S):\mbox{QF}(\hat R_0)]=e(\hat W^*/\hat W)f(\hat W^*/\hat W).
$$
We have that 
$$
e(\hat W^*/\hat W)=e(W^*/W)=s
$$
since $R_0$ and $S$ are analytically unramified (as they are excellent), and
$$
f(\hat W^*/\hat W)=[\mbox{QF}(\hat S/x\hat S):\mbox{QF}(\hat R_0/\mathfrak p\hat R_0)]
=\frac{ad}{s}[S/m_S:R/m_R].
$$
Thus
$$
[\mbox{QF}(\hat S):\mbox{QF}(\hat R_0)]=ad[S/m_S:R/m_R].
$$

Now suppose that $b=0$. Then taking $R_0=R$, $W=R_{uR}$ and $W^*=S_{xS}$, a simpler variant of the above proof shows that
$e(\hat W^*/\hat W)=a$, $f(\hat W^*/\hat W)=d[S/m_S:R/m_R]$ and 
$$
[\mbox{QF}(\hat S):\mbox{QF}(\hat R)]=ad[S/m_S:R/m_R].
$$

\end{proof}

\begin{Proposition}\label{Prop11}   Suppose that $R$ and $S$ are two dimensional excellent regular local rings, such that $S$ dominates $R$ and 
$K^*=\mbox{QF}(S)$ is a finite separable extension of $K=\mbox{QF}(R)$, $R$ has a regular system of parameters $u,v$ and $S$ has a regular system of parameters $x,y$ such that there is an expression
$$
u=\gamma_1 x^ay^c, v=\gamma_2 x^by^d
$$
where  $\gamma_1$, $\gamma_2$ are  units in $S$ and $ad-bc\ne 0$.
Then there exist inclusions
$$
R\rightarrow R_0\rightarrow S
$$
of local rings  with the following properties: $R_0$ is a two dimensional  normal local ring of $K$ which is essentially of finite type over $R$  such that $S$ lies above $R_0$ and 
$$
[\mbox{QF}(\hat S):\mbox{QF}(\hat R_0)]=|ad-bc|[S/m_S:R/m_R].
$$
\end{Proposition}

\begin{proof} The conclusions of this proposition follow from Proposition \ref{Prop10} if one of $a,b,c,d$ is zero, so we may assume that $a,b,c,d$ are all positive. After possibly interchanging $x$ and $y$, it can be assumed that $ad-bc>0$. The proof of this proposition is a generalization of the proof of Proposition \ref{Prop10}, and we give an outline of the proof, indicating the essential differences.

Let $s=\mbox{gcd}(a,b)$ and $s'=\mbox{gcd}(c,d)$. Let $I\subset R$ (respectively $I'\subset R)$ be the integral closure of the the ideal $(u^{\frac{b}{s}}, v^{\frac{a}{s}})$ (respectively $(u^{\frac{d}{s'}}, v^{\frac{c}{s'}})$).
We have that
$$
R_0=R\left[\frac{II'}{u^{\frac{b}{s}}v^{\frac{c}{s'}}}\right]_{m_S\cap R\left[\frac{II'}{u^{\frac{b}{s}}v^{\frac{c}{s'}}}\right]}
$$
is a normal local ring (products of integrally closed ideals in a 2-dimensional regular local ring are integrally closed by Theorem 2', Appendix 5 \cite{ZS2}).

Let $\mathfrak p=xS\cap R_0$. Let $W^*$ be the DVR $W^*= S_{(x)}$ and let $W$ be the DVR $W= (R_0)_{\mathfrak p}=W^*\cap K$.
Let 
$$
\mathfrak p'=(xS)\cap R\left[\frac{I}{u^{\frac{b}{s}}}\right]
$$
and let
$$
\phi=\frac{v^{\frac{a}{s}}}{u^{\frac{b}{s}}}.
$$
$\phi\not \in\mathfrak p'$, so $\mathfrak p'$ is a height one prime in $R\left[\frac{I}{u^{\frac{b}{s}}}\right]$. Thus
$$
W''=R\left[\frac{I}{u^{\frac{b}{s}}}\right]_{\mathfrak p'}
$$
is a DVR dominated by $W^*$. Thus $W=W''$. As in the proof of Proposition \ref{Prop10}, we conclude that 
$$
e(W^*/W)=s
$$
and that 
$$
R_0/\mathfrak p\cong (R/m_R[\overline \phi])_{(\overline \phi)}
$$
where $\overline \phi$ is the  residue of $\phi$ in $R_0/\mathfrak p$. We have that
$$
\overline\nu(\phi)=\frac{ad-bc}{s}
$$
where $\overline\nu$ is the natural valuation of $S/xS$,
and $R_0/m_{R_0}\cong R/m_R$. We calculate (as in the proof of Proposition \ref{Prop10}) that
$$
f(\hat W^*/\hat W)=[\mbox{QF}(\widehat{S/xS}):\mbox{QF}(\widehat{R_0/\mathfrak p})]=\frac{ad-bc}{s}[S/m_S:R/m_R]
$$
and that
$$
[\mbox{QF}(\hat S):\mbox{QF}(\hat R_0)]=e(\hat W^*/\hat W)f(\hat W^*/\hat W)=(ad-bc)[S/m_S:R/m_R].
$$

\end{proof}

\subsection{Rational rank 1}

Throughout  this subsection  we will have the following assumptions. Suppose that $R$ is a 2 dimensional excellent local domain with quotient field $\mbox{QF}(R)=K$. Further suppose that $K^*$ is a finite separable extension of $K$ and $S$ is a 2 dimensional local domain with quotient field
$\mbox{QF}(S)=K^*$ such that  $S$ dominates $R$. 

Suppose that $\nu^*$ is a valuation of $K^*$ such that 
\begin{enumerate}
\item[1)] $\nu^*$ dominates $S$.
\item[2)] The residue field $V_{\nu^*}/m_{\nu^*}$ of $V_{\nu^*}$ is algebraic over $S/m_S$. 
\item[3)] The value group $\Gamma_{\nu^*}$ of $\nu^*$ has rational rank 1 (so it is isomorphic as an ordered group to a subgroup of $\QQ$).
\end{enumerate}
Let $\nu$ be the restriction of $\nu^*$ to $K$.
\vskip .2truein

\begin{Theorem}\label{Theorem2}
Suppose that  the defect $\delta(\nu^*/\nu)=0$. Then there exists a commutative diagram  
\begin{equation}\label{eq522}
\begin{array}{ccccc}
R_1&\rightarrow& S_1&\subset & V_{\nu^*}\\
\uparrow&&\uparrow&&\\
R&\rightarrow &S&&
\end{array}
\end{equation}
such that the vertical arrows are products of quadratic transforms along $\nu^*$ and
\begin{enumerate}
\item[1)] $R_1$ and $S_1$ are two dimensional regular local rings.
\item[2)] $R_1$ has a regular system of parameters $u_1,v_1$ and $S_1$ has a regular system of parameters
$x_1,y_1$ and there exists a unit $\gamma_1 \in S_1$ such that
$$
u_1=\gamma_1 x_1^e\mbox{ and }v_1=y_1
$$
where $e=e(\nu^*/\nu)=|\Gamma_{\nu^*}/\Gamma_{\nu}|$,
and the class of $\nu^*(x_1)$ is a generator of the group $\Gamma_{\nu^*}/\Gamma_{\nu}\cong\ZZ_e$.
\item[3)] $V_{\nu^*}/m_{\nu^*}$ is the join $V_{\nu^*}/m_{\nu^*}=(V_{\nu}/m_{\nu})(S_1/m_{S_1})$ and
$[S_1/m_{S_1}:R_1/m_{R_1}]=f(\nu^*/\nu)=[V_{\nu^*}/m_{\nu^*}:V_{\nu}/m_{\nu}]$.
\end{enumerate}
\end{Theorem}

It will follow from our proof  that there exists a diagram (\ref{eq522}) such that the conditions 1), 2) and 3) of Theorem \ref{Theorem2} are stable under further appropriate sequences of quadratic transforms above $R$ and $S$.

\begin{Proposition}\label{Prop32}
There exists a local ring $R'$  of $K$ which is essentially of finite type over $R$, is dominated by 
$\nu$ and dominates $R$ such that if we have a commutative diagram
\begin{equation}\label{eq31}
\begin{array}{lll}
V_{\nu}&\rightarrow&V_{\nu^*}\\
\uparrow&&\uparrow\\
R_1&\rightarrow&S_1\\
\uparrow&&\\
R'&&\uparrow\\
\uparrow&&\\
R&\rightarrow&S
\end{array}
\end{equation}
where 
$R_1$ is a regular local ring of $K$ which is essentially of finite type over $R$ and dominates $R$, $S_1$ is a regular local ring of $K^*$ which is essentially of finite type over $S$ and dominates $S$, 
$R_1$ has a regular system of parameters $u,v$ and $S_1$ has a regular system of parameters $x,y$ such that there is an expression
$$
u=\gamma x^a, v=x^bf
$$
where $a>0$, $b\ge 0$, $\gamma$ is a unit in $S$, $x\not\,\mid f$ in $S_1$ and $f$ is not a unit in $S_1$, then
\begin{equation}\label{eq37}
ad[S_1/m_{S_1}:R_1/m_{R_1}]=e(\nu^*/\nu)f(\nu^*/\nu)p^{\delta(\nu^*/\nu)}.
\end{equation}
where $d=\overline\nu(f\mbox{ mod }x)$ with $\overline \nu$ being the natural valuation of the DVR $S/xS$.
\end{Proposition}

\begin{proof} We first prove the proposition with the assumption that $K^*/K$ is Galois with Galois group $G$. 

Let $g$ be the number of extensions of $\nu$ to $K^*$.
Writing $f(\nu^*/\nu)=f_0p^s$ where $\mbox{gcd}(f_0,p)=1$ and $e(\nu^*/\nu)=e_0p^t$ where $\mbox{gcd}(e_0,p)=1$, we have 
by the Corollary to Theorem 25, page 78 \cite{ZS2} that
$[G:G^s(\nu^*/\nu)]=g$, $[G^s(\nu^*/\nu):G^i(\nu^*/\nu)]=f_0$,
$[G^i(\nu^*/\nu):G^r(\nu^*/\nu)]=e_0$ and $|G^r(\nu^*/\nu)|=p^{s+t+\delta(\nu^*/\nu)}$,
so that
$$
[K^*:K]=ge(\nu^*/\nu)f(\nu^*/\nu)p^{\delta(\nu^*/\nu)}.
$$
Since $[K^s(\nu^*/\nu):K]=g$, we have that
\begin{equation}\label{eq35}
[K^*:K^s(\nu^*/\nu)]=e(\nu^*/\nu)f(\nu^*/\nu)p^{\delta(\nu^*/\nu)}.
\end{equation}
By the argument at the top of page 86 in \cite{RTM}, there exists $R'$ as above such that for any diagram 
\begin{equation}
\begin{array}{lll}
V_{\nu}&\rightarrow&V_{\nu^*}\\
\uparrow&&\uparrow\\
R^*&\rightarrow&S^*\\
\uparrow&&\uparrow\\
R'&&\\
\uparrow&&\uparrow\\
R&\rightarrow&S
\end{array}
\end{equation}
where $R^*$ and $S^*$ are normal  local rings of $K$ and $K^*$ respectively  such that $R^*$ is essentially of finite type over $R'$ and dominates $R'$ and  $S^*$ lies over $R^*$, we have that
\begin{equation}\label{eq34}
G^s(S^*/R^*)=G^s(\nu^*/\nu).
\end{equation}

Now applying Proposition \ref{Prop10} to an extension $R_1\rightarrow S_1$ satisfying (\ref{eq31}) (so that $R_1$ dominates $R'$), we have a commutative diagram
$$
\begin{array}{lll}
V_{\nu^*}&\rightarrow&V_{\nu}\\
\uparrow&&\uparrow\\
R^*&\rightarrow &S_1\\ 
\uparrow&\nearrow&\\
R_1&&
\end{array}
$$
such that $R^*$ and  $S_1$ are normal  local rings such that $S_1$ lies over $R^*$ and
\begin{equation}\label{eq33}
[\mbox{QF}(\hat S_1):\mbox{QF}(\hat R^*)]=ad[S_1/m_{S_1}:R_1/m_{R_1}].
\end{equation}

Let $S'=K^s(\nu^*/\nu^*)\cap S_1$. Then $R^*\rightarrow S'$ is unramified with $S'/m_{S'}=R^*/m_{R^*}$ by Theorem 1.47 \cite{RTM}, since $K^s(\nu^*/\nu^*)=K^s(S_1/R^*)$
by (\ref{eq34}).  Thus $\hat R^*=\hat S'$ by (10.1) \cite{RES}. We have that 
\begin{equation}\label{eq36}
[K^*:K^s(\nu^*/\nu)]=[\mbox{QF}(\hat S'):\mbox{QF}(\hat R^*)]
\end{equation}
by II of Proposition 1 (page 498) \cite{LU}, since there is a unique local ring in $K^*$ lying over $S'$.
Now combining (\ref{eq35}), (\ref{eq36}) and (\ref{eq33}), we obtain formula (\ref{eq37}). (The proof in \cite{LU} is valid in our more general situation since $S'$ is excellent.)

We will now establish the proposition in the general case, when $K^*/K$ is only assumed to be finite and separable. Let $K'$ be a Galois closure of $K^*/K$, and let $\nu'$ be an extension of $\nu^*$ to $K'$. By the Galois case, there exists a normal local ring $R'$ of $K$ 
giving the property of Proposition \ref{Prop32} within the extension $K'/K$, and a
 normal local ring $S'$ of $K^*$ giving the property of Proposition \ref{Prop32} within the extension $K'/K^*$. We can choose $R'$ and $S'$ so that $S'$ lies  over $R'$.

Now suppose that we are given a diagram (\ref{eq31}). We then have that $S'\subset S_1$.  We have by Proposition \ref{Prop10} local rings $R_0$ of $K$ and $S_1$ of $K^*$ such that $S_1$ lies over $R_0$ and
\begin{equation}\label{eq38}
[\mbox{QF}(\hat S_1):\mbox{QF}(\hat R_0)]=ad[S_1/m_{S_1}:R_1/m_{R_1}].
\end{equation}
Let $T$ be the integral closure of $S_1$ in $K'$ and let $T'=T_{m_{\nu'}\cap T}$. 
By (\ref{eq35}) and  (\ref{eq36}), we have that
$$
[\mbox{QF}(\hat T'):\mbox{QF}(\hat R_0)]=e(\nu'/\nu)f(\nu'/\nu)p^{\delta(\nu'/\nu)}
$$
and
$$
[\mbox{QF}(\hat T'):\mbox{QF}(\hat S_1)]=e(\nu'/\nu^*)f(\nu'/\nu^*)p^{\delta(\nu'/\nu^*)}.
$$ 
Thus
$$
[\mbox{QF}(\hat S_1):\mbox{QF}(\hat R_0)]=\frac{[\mbox{QF}(\hat T'):\mbox{QF}(\hat R_0)]}{[\mbox{QF}(\hat T'):\mbox{QF}(\hat S_1)]}
=e(\nu^*/\nu)f(\nu^*/\nu)p^{\delta(\nu^*/\nu)}
$$
since $e,f$ and $p^{\delta}$ are multiplicative. Now formula (\ref{eq37}) follows from (\ref{eq38}).
\end{proof}

\begin{Proposition}\label{Prop40}
There exists a  local ring $R''$ of $K$ which is essentially of finite type over $R$, is dominated by 
$\nu$ and dominates $R$ such that if we have a commutative diagram
\begin{equation}\label{eq41}
\begin{array}{lll}
V_{\nu}&\rightarrow&V_{\nu^*}\\
\uparrow&&\uparrow\\
R_1&\rightarrow&S_1\\
\uparrow&&\uparrow\\
R''&&\\
\uparrow&&\uparrow\\
R&\rightarrow&S
\end{array}
\end{equation}
where 
$R_1$ is a regular local ring of $K$ which is essentially of finite type over $R''$, $S_1$ is a regular local ring of $K^*$ which is essentially of finite type over $S$ and dominates $S$,
$R_1$ has a regular system of parameters $u,v$ and $S_1$ has a regular system of parameters $x,y$ such that there is an expression
$$
u=\gamma x^a, v=x^bf
$$
where $a>0$, $b\ge 0$, $\gamma$ is a unit in $S$, $x\not\,\mid f$ in $S_1$ and $f$ is not a unit in $S_1$ and there exists a unit $\tau\in S_1$ and $n\in \ZZ_{+}$ such that $h=\tau x^n$, then
\begin{enumerate}
\item[1)] $\nu^*(x)$ is a generator of $\Gamma_{\nu^*}/\Gamma_{\nu}\cong Z_{e(\nu^*/\nu)}$,
\item[2)] $V_{\nu^*}/m_{\nu^*}=(V_{\nu}/m_{\nu})(S_1/m_{S_1})$,
\item[3)]
$$
ad[S_1/m_{S_1}:R_1/m_{R_1}]=e(\nu^*/\nu)f(\nu^*/\nu)p^{\delta(\nu^*/\nu)}
$$
\end{enumerate}
where $d=\overline(f\mbox{ mod }x)$ with $\overline\nu$ being the natural valuation of the DVR $S/xS$.
\end{Proposition}

\begin{proof} Let $R'$ be the  local ring of the conclusions of Proposition \ref{Prop32} and let $g_1,\ldots,g_s\in V_{\nu^*}$ be such that the classes of $g_1,\ldots,g_s$ in $V_{\nu^*}/m_{\nu^*}$ are a $V_{\nu}/m_{\nu}$ basis.   Let $R''$ be a regular   local ring of $K$ which dominates $R'$ and is essentially of finite type over $R'$
such that $g_1,\ldots,g_s$ are in the integral closure of $R''$ in $K^*$.

Suppose that we have a diagram (\ref{eq41}). Then $g_1,\ldots,g_s\in S_1$ so conclusion 2) holds. 

Since $\nu^*(h)=n\nu^*(x)$ is a generator of $\Gamma_{\nu^*}/\Gamma_{\nu}$ we have that $\nu^*(x)$ is a generator of  $\Gamma_{\nu^*}/\Gamma_{\nu}$.

Finally, 3) holds by Proposition \ref{Prop32}, since $R''$ dominates $R'$.
\end{proof}

\begin{Corollary}\label{Prop41} Let assumptions be as in Proposition  \ref{Prop40}, and further assume that $\delta(\nu^*/\nu)=0$. Let $R''$ be the local ring of the conclusions of Proposition \ref{Prop40}. Suppose that we have a commutative diagram (\ref{eq41}). Then 
$$
a=e(\nu^*/\nu), d=1 \mbox{ and } [S_1/m_{S_1}:R_1/m_{R_1}]=f(\nu^*/\nu).
$$
\end{Corollary}

\begin{proof}  We have that $a\nu^*(x)\in \Gamma_{\nu}$ and $\nu^*(x)$ is a generator of $\Gamma_{\nu^*}/\Gamma_{\nu}\cong \ZZ_{e(\nu^*/\nu)}$ by 1) of Proposition \ref{Prop40}
so $e(\nu^*/\nu)$ divides $a$, and thus
\begin{equation}\label{eq43}
a\ge e(\nu^*/\nu).
\end{equation}
Further,
\begin{equation}\label{eq44}
[S_1/m_{S_1}:R_1/m_{R_1}]\ge f(\nu^*/\nu)
\end{equation}
since $V_{\nu^*}/m_{\nu^*}=(V_{\nu}/m_{\nu})(S_1/m_{S_1})$ by 2) of Proposition \ref{Prop40}, so $S_1/m_{S_1}$ contains a basis of $V_{\nu^*}/m_{\nu^*}$ over $V_{\nu}/m_{\nu}$.
Further, we have that
\begin{equation}\label{eq45}
ad[S_1/m_{S_1}:R_1/m_{R_1}]=e(\nu^*/\nu)f(\nu^*/\nu)
\end{equation}
by 3) of Proposition \ref{Prop40} since $\delta(\nu^*/\nu)=0$. Thus by equations (\ref{eq43}), (\ref{eq44}) and (\ref{eq45}), we have 
$$
a[S_1/m_{S_1}:R_1/m_{R_1}]\ge e(\nu^*/\nu)f(\nu^*/\nu)=ad[S_1/m_{S_1}:R_1/m_{R_1}]
$$
giving 
the
conclusions of the corollary.
\end{proof}

\vskip .2truein
We now give the proof of Theorem \ref{Theorem2}. Let $R''$ be the ring of the conclusions of Proposition \ref{Prop40}. We may assume, 
after replacing $R$ and $S$ with appropriate sequences of quadratic transforms of $R$ and $S$, that $R$ and $S$ are regular,
$R$ dominates $R''$ and $R$ has regular parameters $u,v$ and $S$ has regular parameters $x,y$ such that
\begin{equation}\label{eq51}
u=\gamma x^a, v=x^by
\end{equation}
and $h=\tau x^n$ where $h\in V_{\nu^*}$ is such that $\nu^*(h)$ is a generator of $\Gamma_{\nu^*}/\Gamma_{\nu}$ and $\tau$ is a unit in $S$.
(We have that $d=1$ by Corollary \ref{Prop41}). Let $S\rightarrow S_1$ be the smallest sequence of quadratic transforms 
along $\nu^*$ such that $\sqrt{xyS_1}$ is a prime ideal (this is possible since $\nu^*$ has rational rnak 1 so $\nu^*(x)$ and $\nu^*(y)$ are rationally dependent),
 and let $R\rightarrow R_1$ be the smallest sequence of quadratic transforms along $\nu$ such that $\sqrt{uvR_1}$ is a prime ideal. We will show that
$S_1$ dominates $R_1$ and we have regular parameters $u_1, Q$ in $R_1$ and $x_1,Q$ in $S_1$ and a unit $\gamma'$ in $S_1$ such that 
$u_1=\gamma'x_1^{e(\nu^*/\nu)}$. Since these conclusions will then hold under further sequences of quadratic transforms, we will then have established the conclusions of the theorem, and the remark on stability following the statement of Theorem \ref{Theorem2}.

We have that

$$
S_1=S[x_1,y_1]_{m_{\nu^*}\cap S[x_1,y_1]}
$$
 where
\begin{equation}\label{eq52}
x=x_1^{m_1}y_1^{m_1'}, y=x_1^{n_1}y_1^{n_1'}
\end{equation}
with $m_1n_1'-n_1m_1'=\pm 1$, $\nu^*(x_1)>0$ and $\nu^*(y_1)=0$. 
$$
S[x_1,y_1]/x_1S[x_1,y_1]\cong S/m_S[y_1]
$$
is a polynomial ring over $S/m_S$. Let $f\in S/m_S[y_1]$ be the monic generator of 
$$
(m_{\nu^*}\cap S[x_1,y_1])/x_1S[x_1,y_1].
$$
There exists $P\in S[x_1,y_1]$ such that the residue of $P$ in $S[x_1,y_1]/x_1S[x_1,y_1]$ is $f$. Then $x_1, P$ are regular parameters in $S_1$.

Now substitute (\ref{eq52}) into (\ref{eq51}), to obtain
$$
\begin{array}{lll}
u&=& \gamma x_1^{m_1a}y_1^{m_1'a}\\
v&=& x_1^{m_1b+n_1}y_1^{m_1'b+n_1'}.
\end{array}
$$
Let $s=\mbox{gcd}(m_1a, m_1b+n_1)$. By a sequence of substitutions
$\hat u_0=u, \hat v_0=v$, and 
$$
\hat u_i=\hat u_{i+1}, \hat v_i=\hat u_{i+1}\hat v_{i+1}
$$
or
$$
\hat u_i=\hat u_{i+1}\hat v_{i+1}, \hat v_i=\hat v_{i+1}
$$
for $0\le i\le \lambda$ we obtain an expression
$$
\begin{array}{lll}
\hat u_{\lambda}&=& \gamma^{c_1}x_1^sy_1^{t_1}\\
\hat v_{\lambda}&=& \gamma^{c_2}x_1^sy_1^{t_2}.
\end{array}
$$
We have that
$$
s|t_1-t_2|=\left|\mbox{Det}\left(\begin{array}{cc}s&t_1\\s&t_2\end{array}\right)\right|
=\left|\left(\begin{array}{cc} m_1a& m_1'a\\ (m_1b+n_1)& (m_1'b+n_1')\end{array}\right)\right|
=a\left|\left(\begin{array}{cc} m_1&m_1'\\n_1&n_1'\end{array}\right)\right|=a.
$$
We thus have that $t_2-t_1\ne 0$. Set
$$
u_1=\hat u_{\lambda}, v_1=\frac{\hat v_{\lambda}}{\hat u_{\lambda}}\mbox{ if  }t_2-t_1>0,
$$
$$
u_1=\frac{\hat u_{\lambda}}{\hat v_{\lambda}}, v_1=\hat v_{\lambda}\mbox{ if }t_2-t_1<0.
$$
Without loss of generality, we may assume that $t_2-t_1>0$. 
We then have an expression
$$
u=u_1^{m_2}v_1^{m_2'}, v= u_1^{n_2}v_1^{n_2'}
$$
with $m_2n_2'-n_2m_2'=\pm 1$, $\nu(u_1)>0$ and $\nu(v_1)=0$. Now
$$
R_1=R[u_1,v_1]_{m_{\nu}\cap R[u_1,v_1]}.
$$
We have that $R[u_1,v_1]/u_1R[u_1,v_1]\cong R/m_R[v_1]$ is a polynomial ring over $R/m_R$. Let $g\in R/m_R[v_1]$ be the monic generator of 
$$
(m_{\nu}\cap R[u_1,v_1])/u_1R[u_1,v_1].
$$
There exists $Q\in R[u_1,v_1]$ such that the residue of $Q$ in $R[u_1,v_1]/u_1R[u_1,v_1]$ is $g$. $u_1,Q$ are regular parameters in $R_1$. 
We have an expression
$$
u_1=(\gamma^{c_1}y_1^{t_1})x_1^s, v_1=\gamma^{c_2-c_1}y_1^{t_2-t_1}.
$$
The inclusion $R\rightarrow S$ induces a natural homomorphism
$$
R/m_R[v_1]\cong R[u_1,v_1]/u_1R[u_1,v_1]\rightarrow S[x_1,y_1]/x_1S[x_1,y_1]\cong S/m_S[y_1].
$$
Let $0\ne \gamma_0$ be the residue of $\gamma$ in $S/m_S$. Then $\gamma_0$ is the residue of $\gamma$ in $S/m_S[y_1]$ and
$\gamma_0^{c_2-c_1}y_1^{t_2-t_1}$ is the residue of $v_1$ in $S/m_S[y_1]$. The residue of $Q$ in $S/m_S[y_1]$ is $g(\gamma_0y_1^{t_2-t_1})$ which is nonzero. Thus $b(S_1/R_1)=0$. We further have that  $a(S_1/R_1)=e(\nu^*/\nu)$, $d(S_1/R_1)=1$ and $[S_1/m_{S_1}:R_1/m_{R_1}]=f(\nu^*/\nu)$ by Proposition \ref{Prop40}. This completes the proof of Theorem \ref{Theorem2}.
\vskip .2truein

\subsection{Rational rank 2}

We have the following Theorem \ref{Theorem5} for rational rank 2 valuations dominating a two dimensional excellent local domain, most of whose simple proof is as on page 40 of \cite{CP}. The statement that the extension is defectless follows from the argument of Proposition \ref{Prop32} which establishes (\ref{eq37}), replacing references to Proposition \ref{Prop10} with Proposition \ref{Prop11}. We obtain the formula
$$
|ad-bc|[S_1/m_{S_1}:R_1/m_{R_1}]=e(\nu^*/\nu)f(\nu^*/\nu)p^{\delta(\nu^*/\nu)}
$$
instead of (\ref{eq37}). Since $e(\nu^*/\nu)=|ad-bc|$, we obtain that $\delta(\nu^*/\nu)=0$.
 Related results are proven for valuations of maximal rational rank in algebraic function fields in Theorem 3.1 \cite{KK2} 

\begin{Theorem}\label{Theorem5} Suppose that $R$ is a 2 dimensional excellent local domain with quotient field $\mbox{QF}(R)=K$. Further suppose that $K^*$ is a finite separable extension of $K$ and $S$ is a 2-dimensional local domain with quotient field
$\mbox{QF}(S)=K^*$ such that  $S$ dominates $R$. 

Suppose that $\nu^*$ is a valuation of $K^*$ such that 
\begin{enumerate}
\item[1)] $\nu^*$ dominates $S$.
\item[2)] The residue field $V_{\nu^*}/m_{\nu^*}$ of $V_{\nu^*}$ is algebraic over $S/m_S$. 
\item[3)] The value group $\Gamma_{\nu^*}$ of $\nu^*$ has rational rank 2. 
\end{enumerate}
Let $\nu$ be the restriction of $\nu^*$ to $K$.
 Then the defect $\delta(\nu^*/\nu)=0$, and there exists a commutative diagram  
\begin{equation}\label{eq622}
\begin{array}{ccccc}
R_1&\rightarrow& S_1&\subset & V_{\nu^*}\\
\uparrow&&\uparrow&&\\
R&\rightarrow &S&&
\end{array}
\end{equation}
such that the vertical arrows are products of quadratic transforms along $\nu^*$ and
\begin{enumerate}
\item[1)] $R_1$ and $S_1$ are two dimensional regular local rings.
\item[2)] $R_1$ has a regular system of parameters $u_1,v_1$ and $S_1$ has a regular system of parameters
$x_1,y_1$ and there exist units $\gamma_1,\tau_1 \in S_1$ such that
$$
u_1=\gamma_1 x_1^ay_1^b\mbox{ and }v_1=\tau_1x_1^cy_1^d
$$
where 
$$
e=e(\nu^*/\nu)=|\Gamma_{\nu^*}/\Gamma_{\nu}|=\left|\mbox{Det}\left(\begin{array}{cc} a&b\\ c&d\end{array}\right)\right|
$$
and the classes of $\nu^*(x_1), \nu^*(y_1)$ are generators of the group $\Gamma_{\nu^*}/\Gamma_{\nu}\cong\ZZ^2/A\ZZ^2$.
\item[3)] $V_{\nu^*}/m_{\nu^*}$ is the join $V_{\nu^*}/m_{\nu^*}=(V_{\nu}/m_{\nu})(S_1/m_{S_1})$ and
$[S_1/m_{S_1}:R_1/m_{R_1}]=f(\nu^*/\nu)=[V_{\nu^*}/m_{\nu^*}:V_{\nu}/m_{\nu}]$.
\end{enumerate}
\end{Theorem}
\vskip .2truein
There exists a diagram (\ref{eq622}) such that the  conditions 1), 2) and 3) of Theorem \ref{Theorem5} are stable under further appropriate sequences of quadratic transforms above $R$ and $S$.

\subsection{Proof of Theorem \ref{Theorem7}}

\vskip .2truein
We now give the proof of Theorem \ref{Theorem7}.  
$$
\mbox{trdeg}_{S/m_S}V_{\nu^*}/m_{\nu^*}+\dim_{\QQ}\Gamma_{\nu^*}\otimes \QQ\le 2
$$
by Abhyankar's inequality (Theorem 1 \cite{Ab1}). 

If $\mbox{trdeg}_{R/m_R}S/m_S=1$, then $V_{\nu^*}$ and $V_{\nu}$ are DVRs and algebraic local rings of $K^*$ and $K$ respectively, so that $V_{\nu}\rightarrow V_{\nu^*}$ is a monomial mapping, which is obtained from $R$, respectively $S$ by a sequence of quadratic transforms along $\nu^*$ (there exists a sequence of quadratic transforms $R\rightarrow R^*$ along $\nu$ such that $V_{\nu}/m_{\nu}$ is algebraic over $R/m_R$. Thus $R^*=V_{\nu}$ be Zariski's Main Theorem, (10.7) \cite{RES}). (For this type of valuation  we must have that $\delta(\nu^*/\nu)=0$
by Theorem 20, Section 11, Chapter VI \cite{ZS2}).

Suppose that $\dim_{\QQ}\Gamma_{\nu^*}\otimes \QQ=2$. Then $V_{\nu^*}/m_{\nu^*}$ is algebraic over $S/m_S$ by Abhyankar's inequality. In this case there exists a monomialization by Theorem \ref{Theorem5} (we also always have that $\delta(\nu^*/\nu)=0$ for this type of valuation by Theorem \ref{Theorem5}).

The remaining case is when $V_{\nu^*}/m_{\nu^*}$ is algebraic over $S/m_S$, $\dim_{\QQ}\Gamma_{\nu^*}\otimes\QQ=1$
and $\delta(\nu^*/\nu)=0$. The existence of a monomialization in this case follows from Theorem \ref{Theorem2}.

\section{Extensions of Associated Graded Rings of Valuations} 
In this section, we extend the results of \cite{GHK} and \cite{GK} calculating the extension of associated graded rings of a valuation for defectless extensions of 2 dimensional algebraic function fields over an algebraically closed field, and of \cite{CV1} for 2 dimensional algebraic function fields over a (not necessarily closed) characteristic zero field.
We refer to the introduction of this paper for a discussion of this problem.

We recall the following theorem on strong monomialization of rational rank 1 valuations in an extension of characteristic zero function fields from \cite{C}.

\begin{Theorem}\label{Theorem1}(The rational rank 1 case of Theorem 5.1 \cite{C}, Theorem 6.1 \cite{CP} and Theorem 6.5 \cite{CG})
Let $K$ an algebraic function field over a field $k$ of characteristic zero, $K^*$ a finite algebraic extension of $K$ and $\nu^*$ a rational rank 1  $k$ valuation of $K^*$. Suppose that $S^*$ is an algebraic local ring with quotient field $K^*$ which is dominated by $\nu^*$ and $R^*$ is an algebraic local ring with quotient field $K$ which is dominated by $S^*$. Let $\nu$ be the restriction of $\nu^*$ to $K$. Then there exists a commutative diagram  
\begin{equation}\label{eq2}
\begin{array}{ccccc}
R&\rightarrow& S&\subset & V_{\nu^*}\\
\uparrow&&\uparrow&&\\
R^*&\rightarrow &S^*&&
\end{array}
\end{equation}
such that the vertical arrow are product of monoidal transforms along $\nu^*$ and
\begin{enumerate}
\item[1)] $R$ and $S$ are regular local rings of dimension equal to $n=\mbox{trdeg}_kV_{\nu^*}/m_{\nu^*}$.
\item[2)] $R$ has a regular system of parameters $x_1,\ldots,x_n$ and $S$ has a regular system of parameters
$y_1,\ldots,y_n$ and there exists a unit $\delta \in S$ such that
$$
x_1=\delta y_1^e\mbox{ and }x_i=y_i\mbox{ for }i\ge 2
$$
where $e=e(\nu^*/\nu)=|\Gamma_{\nu^*}/\Gamma_{\nu}|$.
\item[3)] The class of $\nu(y_1)$ is a generator of the group $\Gamma_{\nu^*}/\Gamma_{\nu}$.
\item[4)]  $V_{\nu^*}/m_{\nu^*}=(V_{\nu}/m_{\nu})(S/m_{S})$.
\item[5)] $[S/m_S:R/m_R]=f=f(\nu^*/\nu)=[V_{\nu^*}/m_{\nu^*}:V_{\nu}/m_{\nu}]$ 
\item[6)] The  conclusions of  1) and 3) - 5) of the theorem continue to hold for $R_1\rightarrow S_1$ whenever
there exists a commutative diagram  
\begin{equation}\label{eq3}
\begin{array}{ccccc}
R_1&\rightarrow& S_1&\subset & V_{\nu^*}\\
\uparrow&&\uparrow&&\\
R&\rightarrow &S&&
\end{array}
\end{equation}
such that the vertical arrow are product of monoidal transforms along $\nu^*$ and 2) holds for $R_1\rightarrow S_1$.
\end{enumerate}
\end{Theorem}

We now show that with  the assumptions of the conclusions of Theorem \ref{Theorem1}, we further have a very simple
description of the extension of associated graded rings of the valuations.

\begin{Theorem}\label{Theorem4} Let assumptions be as in Theorem \ref{Theorem1}, and let $R\rightarrow S$ be as in the conclusions of Theorem \ref{Theorem1}.   Then we have a natural isomorphism of graded rings
$$
{\rm gr}_{\nu^*}(S)\cong \left({\rm gr}_{\nu}(R)\otimes_{R/m_R}S/m_S\right)[Z]/(Z^e-[\delta]^{-1}[x_1]),
$$
where $[\gamma_1], [x_1]$ are the respective classes in ${\rm gr}_{\nu}(R_1)\otimes_{R_1/m_{R_1}}S_1/m_{S_1}$.
The degree of the extension of quotient fields of ${\rm gr}_{\nu}(R)\rightarrow {\rm gr}_{\nu^*}(S)$ is $e(\nu^*/\nu)f(\nu^*/\nu)$.
\end{Theorem}

\begin{proof} 
Let $\gamma_1,\ldots,\gamma_f\in S$ be such that their residues $\overline \gamma_1,\ldots, \overline \gamma_f$ in $S/m_S$ are  a basis of $S/m_S$ over $R/m_R$.
We now establish the following formula:
\begin{equation}\label{eq12}
\mbox{Suppose that $h_1,\ldots, h_f\in R$. Then $\nu^*(\sum_{i=1}^f\gamma_ih_i)=\min\{\nu(h_i)\}$}.
\end{equation}
Without loss of generality, we may suppose that $\nu(h_1)$ is this   minimum. Then $\frac{h_i}{h_1}\in V_{\nu}$ for all $i$ so we have classes
$$
[\frac{h_i}{h_1}]\in V_{\nu}/m_{\nu}\subset V_{\nu^*}/m_{\nu^*}.
$$ 
If $\nu^*(\sum_{i=1}^f\gamma_ih_i)>\nu^*(h_1)$, then
$$
\sum\overline \gamma_i[\frac{h_i}{h_1}]=0\mbox{ in }V_{\nu^*}/m_{\nu^*},
$$
which is impossible since $\overline \gamma_1,\ldots,\overline \gamma_f$ are a basis of $V_{\nu^*}/m_{\nu^*}$
over $V_{\nu}/m_{\nu}$, by the assumptions of the theorem. We have thus established formula
(\ref{eq12}).

Suppose that $z\in S$. Let $t$ be a positive integer such that $t\nu^*(m_S)>\nu^*(z)$. We have an expression
$$
z=\left(\sum_{j=0}^{e-1}\sum_{i=1}^f g_{ij}\gamma_i y_1^j\right)+w
$$
where $g_{ij}\in R$ for all $i,j$ and $w\in m_S^t$.

We will next establish the following formula:
\begin{equation}\label{eq13}
\nu^*(z)=\min\{\nu(g_{ij})+j\nu^*( y_1)\}.
\end{equation}

With our assumption on $t$,
$$
\nu^*(z)=\nu^*\left(\sum_{j=0}^{e-1}\sum_{i=1}^f g_{ij}\gamma_i y_1^j\right).
$$
Set 
$$
h_j=\sum_{i=1}^fg_{ij}\gamma_i\mbox{ for }0\le j\le e-1.
$$
By (\ref{eq12}), $\nu^*(h_j)\in \Gamma_{\nu}$ for all $j$. For $0\le j\le e-1$, we have
$$
\nu^*(h_k y_1^k)-\nu^*(h_j y_1^j)=(k-j)\nu^*( y_1)+\nu^*(h_k)-\nu^*(h_j).
$$
Since the class of $\nu^*( y_1)$ has order $e$ in $\Gamma_{\nu^*}/\Gamma_{\nu}$, we have that
\begin{equation}\label{eq14}
\nu^*(h_k y_1^k)\ne \nu^*(h_j y_1^j)\mbox{ for }j\ne k.
\end{equation}
Formula (\ref{eq13}) now follows from (\ref{eq14}) and (\ref{eq12}).

For $\gamma\in \Gamma_{\nu}$, we have a homomorphism
$$
\mathcal P_{\gamma}(R)/\mathcal P^+_{\gamma}(R)\otimes_{ R/m_{ R}} S/m_{ S}
\rightarrow
\mathcal P_{\gamma}(S)/\mathcal P^+_{\gamma}(S)
$$
defined by $g_i\otimes\gamma_i\mapsto g_i\gamma_i$. This map is 1-1 by the proof of (\ref{eq12}). Thus
$$
\left(\bigoplus_{\gamma\in\Gamma_{\nu}}\mathcal P_{\gamma}(R)/\mathcal P^+_{\gamma}(R)\right)\otimes_{ R/m_{ R}} S/m_{ S}
$$
is a graded subalgebra of 
$$
\bigoplus_{\tau\in\Gamma_{\nu^*}}\mathcal P_{\tau}(S)/\mathcal P^+_{\tau}(S).
$$
Suppose that $h\in S$. Let $\tau=\nu^*(h)$. By (\ref{eq14}), there exists a unique $i$ with $0\le i\le e-1$ 
and $\gamma\in \Gamma_{\nu}$ such that $\tau=\gamma+i\nu^*( y_1)$. Further, 
$$
h=\left(\sum_{j=1}^f\gamma_jg_j\right) y_1^i+ h_2
$$
where $g_j\in R$ with $\nu(g_j)=\gamma$ for $1\le j\le f$ and $h_2\in S$ satisfies $\nu^*(h_2)>\tau$.
Thus
$$
\mathcal P_{\tau}(S)/\mathcal P^+_{\tau}(S)=
\left(\mathcal P_{\gamma}(R)/\mathcal P^+_{\gamma}(R)\otimes_{ R/m_{ R}} S/m_{ S}\right)[\overline y_1]
$$
where $\overline y_1$ is the class of $y_1$.
We have that 
$$
(\overline y_1)^e-[\delta]^{-1}[x_1]=0
$$
where $[\delta]$, $[x_1]$ are the classes of $\delta$ and $x_1$. All other relations on $\overline y_1$ are divisible by this relation by (\ref{eq13}).

\end{proof}

\begin{Theorem}\label{Theorem3} Let assumptions be as in Theorem \ref{Theorem2}, and let  $R_1\rightarrow S_1$ be as in the conclusions of Theorem \ref{Theorem2}.   Then we have a natural isomorphism of graded rings
$$
{\rm gr}_{\nu^*}(S_1)\cong \left({\rm gr}_{\nu}(R_1)\otimes_{R_1/m_{R_1}}S_1/m_{S_1}\right)[Z]/(Z^e-[\gamma_1]^{-1}[u_1]),
$$
where $[\gamma_1], [u_1]$ are the respective classes in ${\rm gr}_{\nu}(R_1)\otimes_{R_1/m_{R_1}}S_1/m_{S_1}$.
The degree of the extension of quotient fields of ${\rm gr}_{\nu}(R_1)\rightarrow {\rm gr}_{\nu^*}(S_1)$ is $e(\nu^*/\nu)f(\nu^*/\nu)$.
\end{Theorem}

\begin{proof} The proof is exactly the same as the proof of Theorem \ref{Theorem4}. 
\end{proof}

We also obtain from Theorem \ref{Theorem5} the following result, showing that even in positive and mixed characteristic, the associated graded rings of  Abhyankar valuations
dominating a stable extension of two dimensional excellent regular local rings have a nice form. 

\begin{Theorem}\label{Theorem6} Let assumptions be as in Theorem \ref{Theorem5}, and let  $R_1\rightarrow S_1$ be as in the conclusions of Theorem \ref{Theorem5}.   Then we have a natural isomorphism of graded rings
$$
{\rm gr}_{\nu^*}(S_1)\cong \left({\rm gr}_{\nu}(R_1)\otimes_{R_1/m_{R_1}}S_1/m_{S_1}\right)[X,Y]/(X^aY^b-[\gamma_1]^{-1}[u_1],
X^cY^d-[\tau_1]^{-1}[v_1]),
$$
where $[\gamma_1], [\tau_1], [u_1], [v_1]$ are the respective classes in ${\rm gr}_{\nu}(R_1)\otimes_{R_1/m_{R_1}}S_1/m_{S_1}$.
The degree of the extension of quotient fields of ${\rm gr}_{\nu}(R_1)\rightarrow {\rm gr}_{\nu^*}(S_1)$ is $e(\nu^*/\nu)f(\nu^*/\nu)$.
\end{Theorem}

\section{Non constancy of $\alpha_n$ and $\beta_n$}\label{SectionCSM}

In Corollary 7.30 and Theorem 7.33 of \cite{CP}, it is shown that $\alpha_n+\beta_n$ is a constant for $n\gg 0$, where  $\alpha_n$ and $\beta_n$ are the integers defined in (\ref{eqI5}) which are associated to the stable forms (\ref{eq100}) of an extension of valued two dimensional algebraic function fields.  If $\Gamma_{\nu}$ is not $p$-divisible, it is 
further shown that $\alpha_n$ and $\beta_n$ are both constant for $n\gg 0$, and $p^{\beta_n}$ is the defect of the extension.
However, if $\Gamma_{\nu}$ is $p$-divisible, then it is only shown that the sum $\alpha_n+\beta_n$ is constant for $n\gg 0$, and that 
$p^{\alpha_n+\beta_n}$ is the defect of the extension. In Remark 7.34 \cite{CP} it is asked if $\alpha_n$ and $\beta_n$ (and some other numbers computed from  the stable forms) are eventually constant in the case when $\Gamma_{\nu}$ is $p$-divisible. We give examples (\ref{eq500}) - (\ref{eq503}) here showing that this is not the case, even within defect Artin Schreier extensions. 

The example in Theorem 7.38 \cite{CP} is a tower $K^*/K$ of two Artin Schreier extensions,
$$
K\rightarrow K_1\rightarrow K^*,
$$
where $K=k(u,v)$, $K_1=k(x,v)$ and $K^*=k(x,y)$, over an algebraically closed field $k$ of characteristic $p>0$, with
$$
u=\frac{x^p}{1-x^{p-1}}, v=y^p-x^cy
$$
where $p-1$ divides $c$. $R$ and $S$ are defined to be $R=k[u,v]_{(u,v)}$ and $S=k[x,y]_{(x,y)}$.
The valuation $\nu^*$ on $K^*$ is defined by the generating sequence (\ref{New1}) in $S$ and the valuation $\nu=\nu^*|K$ is defined by the generating sequence (\ref{New2}) in $R$. In Theorem 7.38 \cite{CP}, it is shown that $\Gamma_{\nu^*}=\Gamma_{\nu}=\frac{1}{p^{\infty}}\ZZ$, $\delta(\nu^*/\nu)=2$, and in the stable forms $R_n\rightarrow S_n$ above $R\rightarrow S$, we have $\alpha_n=1$ and $\beta_n=1$ for all $n$.

Let $\nu_1=\nu^*|K_1$.
We define $A=k[x,v]_{(x,v)}$, a local ring of $K_1$.  
We will show in this section that 
$\delta(\nu^*/\nu_1)=\delta(\nu_1/\nu)=1$ and we have  stable forms $R_{j}\rightarrow A_{j}$ of $K_1/K$ (Theorem \ref{Theorem72}) and $A_j\rightarrow S_j$ of $K^*/K_1$ (Theorem \ref{Theorem71}) such that
by the conclusions of Theorem \ref{Theorem71}, 
\begin{equation}\label{eq500}
\alpha_j(S_j/A_j)=\left\{\begin{array}{ll}
1&\mbox{ if $j$ is even}\\
0&\mbox{ if $j$ is odd}
\end{array}\right.
\end{equation}
and 
\begin{equation}\label{eq501}
\beta_j(S_j/A_j)=\left\{\begin{array}{ll}
0&\mbox{ if $j$ is even}\\
1&\mbox{ if $j$ is odd}
\end{array}\right.
\end{equation}
for $A_j\rightarrow S_j$, and 
by  Theorem \ref{Theorem72}, 
we have 
\begin{equation}\label{eq502}
\alpha_j(A_j/R_j)=\left\{\begin{array}{ll}
0&\mbox{ if $j$ is even}\\
1&\mbox{ if $j$ is odd}
\end{array}\right.
\end{equation}
and
\begin{equation}\label{eq503}
\beta_j(A_j/R_j)=\left\{\begin{array}{ll}
1&\mbox{ if $j$ is even}\\
0&\mbox{ if $j$ is odd}
\end{array}\right.
\end{equation}
for $R_j\rightarrow A_j$.
\vskip .2truein

We now prove these statements. We make use of the notation introduced in Section 7.11 \cite{CP} in the construction of the example of Theorem 7.38 \cite{CP}. The proof makes essential use of the theory of generating sequences of a valuation dominating a two dimensional regular local ring, as developed in \cite{Sp}  and  extended in \cite{CV1}. 

 We define a $k$-valuation $\nu_1$ of $K_1$ by prescribing a generating sequence  in $A$, starting with 
\begin{equation}\label{eq61}
U_0=x, U_1=v, U_2=v^p-x
\end{equation}
and for $j\ge 2$,
\begin{equation}\label{eq60}
\begin{array}{lll}
U_{j+1}&=& U_j^p-x^{p^{2j-2}}U_{j-1}\mbox{ if $j$ is odd}\\
U_{j+1}&=& U_j^{p^3}-x^{p^{2j-1}}U_{j-1}\mbox{ if $j$ is even.}
\end{array}
\end{equation}

We now establish that $\{U_i\}$ determines a unique valuation $\nu_1$ on $K_1$ such that $\nu_1(x)=1$. Define $\overline\gamma_j$ by
$$
\overline\gamma_0=1, \overline\gamma_1=\frac{1}{p}
$$
and for $j\ge 2$,
\begin{equation}\label{eq122}
\overline\gamma_j=\left\{\begin{array}{ll}
\frac{1}{p}(p^{2j-2}+\overline\gamma_{j-1})&\mbox{ if $j$ is odd}\\
\frac{1}{p^3}(p^{2j-1}+\overline \gamma_{j-1})&\mbox{ if $j$ is even}
\end{array}\right.
\end{equation}
By induction on $j\ge 0$, we have that
\begin{equation}\label{eq123}
\overline\gamma_{j+1}=\left\{\begin{array}{ll}
p^{2j-2}(\sum_{j'=0}^j \frac{1}{p^{4j'}})&\mbox{ if $j$ is odd}\\
p^{2j-1}(\sum_{j'=0}^j\frac{1}{p^{4j'}})&\mbox{ if $j$ is even}
\end{array}\right.
\end{equation}
Let
$$
n_i=\left\{\begin{array}{ll}
p&\mbox{ if $j$ is odd}\\
p^3&\mbox{ if $j$ is even}
\end{array}\right.
$$
Let $\Gamma_j$ be the group generated by $\overline\gamma_0,\ldots,\overline\gamma_j$. By Remark 7.171 \cite{CP} or Theorem 1.1 \cite{CV1} and its proof, $\{U_i\}$ is a generating sequence of a unique valuation $\nu_1$ on $K_1$ such that $\nu_1(U_i)=\overline\gamma_i$ for all $i$ if 
\begin{equation}\label{eq120}
n_i=[\Gamma_i:\Gamma_{i-1}]
\end{equation}
for $i\ge 1$ and
\begin{equation}\label{eq121}
\overline\gamma_{i+1}>n_i\overline\gamma_i
\end{equation}
for $i\ge 1$.

By (\ref{eq122}), $n_i\overline\gamma_i\in \Gamma_{i-1}$, so $[\Gamma_i:\Gamma_{i-1}]=n_i$ if $\overline\gamma_i$ has order precisely $n_i$ in $\Gamma_{i}/\Gamma_{i-1}$.

By (\ref{eq123}), we have that
$$
\Gamma_{i-1}=\left\{\begin{array}{ll}
\frac{1}{p^{2i-2}}\ZZ&\mbox{ if $i$ is odd}\\
\frac{1}{p^{2i-3}}\ZZ&\mbox{ if $i$ is even}
\end{array}\right.
$$
for $i\ge 1$, so that by (\ref{eq123}), $\gamma_i$ has order $n_i$ in $\Gamma_{i}/\Gamma_{i-1}$.
Since (\ref{eq121}) holds by (\ref{eq123}), we have that $\{U_i\}$ is a generating sequence in $A$ which determines a valuation $\nu_1$ of $K_1$.

The following Propositions \ref{Prop87} and \ref{Prop88} are  a little stronger than Proposition 7.40 \cite{CP}.

\begin{Proposition}\label{Prop87}
Let 
$$
S_1=S\rightarrow S_2\rightarrow\cdots
$$
be the sequence of quadratic transforms such that $S_i=S_{r_i'}$ in the notation of Definition 7.11 \cite{CP} so that $V_{\nu^*}=\cup S_i$.
Let $Q_i$ be the generating sequence of $S$ of (71) of \cite{CP} determining $\nu^*$ (Proposition 7.40 \cite{CP})
\begin{equation}\label{New1}
\begin{array}{lll}
Q_0&=&x\\
Q_1&=& y\\
Q_2&=& y^{p^2}-x\\
Q_{j+1}&=& Q_j^{p^2}=x^{p^{2j-2}}Q_{j-1}\mbox{ for }j\ge 2.
\end{array}
\end{equation}
Then there exist generating sequences
$\{Q(k)_i\}$ in $S_k$ such that 
$$
x_{S_2}=Q(2)_0=Q_1=y, y_{S_2}=Q(2)_1=\frac{Q_2}{Q_0}=\frac{Q_2}{x}
$$
are regular parameters in $S_2$ such that $Q(2)_0=0$ is a local equation of the exceptional locus of $\mbox{Spec}(S_2)\rightarrow \mbox{Spec}(S)$ and
for $k\ge 2$,
$$
x_{S_{k+1}}=Q(k+1)_0=Q(k)_1, y_{S_{k+1}}=Q(k+1)_1=\frac{Q_{k+1}}{x^{p^{2k-2}}Q_{k-1}}
$$
are regular parameters in $S_{k+1}$, such that $Q(k+1)_0=0$ is a local equation of the exceptional locus of $\mbox{Spec}(S_{k+1})\rightarrow \mbox{Spec}(S)$. 

In $S_2$,  the generating sequence $\{Q(2)_i\}$ is defined by 
\begin{equation}
Q(2)_0=Q_1, Q(2)_i=\frac{Q_{i+1}}{x^{p^{2(i-1)}}}
\end{equation}
for $i\ge 1$, and for $k\ge 3$, the generating sequence $\{Q(k)_j\}$ in $S_k$ is defined by 
\begin{equation}
Q(k)_0=Q(k-1)_1=\frac{Q_{k-1}}{x^{p^{2(k-3)}}Q_{k-3}}
\end{equation}
and
\begin{equation}
Q(k)_i=\frac{Q_{i+k-1}}{x^{p^{2(i+k-3)}}Q_{k-2}^{p^{2(i-1)}}}
\end{equation}
for $i\ge 1$.
\end{Proposition}

It follows from Proposition \ref{Prop87} that $\nu^*(Q(k)_0)=p^2\nu^*(Q(k)_1)$, so we have regular parameters $\overline x_{S_{k+1}}, \overline y_{S_{k+1}}$ in $S_{k+1}$ defined by
\begin{equation}\label{eq89*}
Q(k)_0=\overline x_{S_{k+1}}^{p^2}(\overline y_{S_{k+1}}+1),
Q(k)_1=\overline x_{S_{k+1}}.
\end{equation}

\begin{Proposition}\label{Prop88}
Let 
$$
R_1=R\rightarrow R_2\rightarrow\cdots
$$
be the sequence of quadratic transforms such that $R_i=R_{r_i'}$ in the notation of Definition 7.11 \cite{CP} so that $V_{\nu}=\cup R_i$.
Let $P_i$ be the generating sequence of $R$ of (76) of \cite{CP} determining $\nu$ (Corollary 7.41 \cite{CP})
\begin{equation}\label{New2}
\begin{array}{lll}
P_0&=& u\\
P_1&=& v\\
P_2&=& v^{p^2}-u\\
P_{i+1}&=& P_i^{p^2}-u^{p^{2i-2}}P_{i-1}\mbox{ for }i\ge 2.
\end{array}
\end{equation}
Then there exist generating sequences
$\{P(k)_i\}$ in $R_k$ determining $\nu$ such that 
$$
u_{R_2}=P(2)_0=P_1=v, v_{R_2}=P(2)_1=\frac{P_2}{P_0}=\frac{P_2}{u}
$$
are regular parameters in $R_2$ such that $u_{R_2}=0$ is a local equation of the exceptional locus of $\mbox{Spec}(S_2)\rightarrow\mbox{Spec}(S)$ and
for $k\ge 2$,
$$
u_{R_{k+1}} = P(k+1)_0=P(k)_1, v_{R_{k+1}}=P(k+1)_1=\frac{P_{k+1}}{u^{p^{2k-2}}P_{k-1}}
$$
are regular parameters in $R_{k+1}$, such that $P(k+1)_0=0$ is a local equation of the exceptional locus of $\mbox{Spec}(R_{k+1})\rightarrow \mbox{Spec}(R)$. 

In $R_2$,  the generating sequence $\{P(2)_j\}$ is defined by 
\begin{equation}
P(2)_0=P_1, P(2)_i=\frac{P_{i+1}}{u^{p^{2(i-1)}}}
\end{equation}
for $i\ge 1$, and for $k\ge 3$, the generating sequence $\{P(k)_j\}$ in $R_k$ is defined by 
\begin{equation}
P(k)_0=P(k-1)_1=\frac{P_{k-1}}{u^{p^{2(k-3)}}P_{k-3}}
\end{equation}
and
\begin{equation}
P(k)_i=\frac{P_{i+k-1}}{u^{p^{2(i+k-3)}}P_{k-2}^{p^{2(i-1)}}}
\end{equation}
for $i\ge 1$.
\end{Proposition}

It follows from Proposition \ref{Prop88} that we have regular parameters $\overline u_{R_{k+1}}, \overline v_{R_{k+1}}$ in $R_{k+1}$ defined by
\begin{equation}\label{eq90}
P(k)_0=\overline u_{R_{k+1}}^{p^2}(\overline v_{R_{k+1}}+1),
P(k)_1=\overline u_{R_{k+1}}.
\end{equation}

\begin{Proposition}\label{Prop70} The sequence $\{U_j\}$ is a generating sequence in $A$ of a unique $k$-valuation $\nu_1$ of $K_1$ such that $\nu_1(x)=1$. Let
$$
A_1=A\rightarrow A_2\rightarrow \cdots
$$
be the sequence where $A_i=A_{a_i'}$ in the notation of Definition 7.11 \cite{CP}, so that $V_{\nu_1}=\cup A_i$.
Then there exist generating sequences $\{U(k)_i\}$ in $A_k$ determining $\nu_1$ such that
$$
x_{A_2}= U(2)_0=U_1=v, v_{A_2}=U(2)_1=\frac{U_2}{U_0}=\frac{U_2}{x}
$$
are regular parameters in $A_2$ such that $U(2)_0=0$ is a local equation of the exceptional locus of $\mbox{Spec}(A_2)\rightarrow \mbox{Spec}(A)$.
For $j\ge 2$,
\begin{equation}\label{eq78}
x_{A_{j+1}}=U(j+1)_0=U(j)_1, v_{A_{j+1}}=U(j+1)_1=
\left\{\begin{array}{ll}
\frac{U_{j+1}}{x^{p^{2j-2}}U_{j-1}}&\mbox{ if $j$ is odd}\\
\frac{U_{j+1}}{x^{p^{2j-1}}U_{j-1}}&\mbox{ if $j$ is even}
\end{array}\right.
\end{equation}
are regular parameters in $A_{j+1}$, such that $U(j+1)_0=0$ is a local equation of the exceptional locus of $\mbox{Spec}(A_{j+1})\rightarrow \mbox{Spec}(A)$. 

In $A_2$, the generating sequence $\{U(2)_j\}$ is defined by 
\begin{equation}\label{E1}
U(2)_0=U_1
\end{equation}
and for $j\ge 1$,
\begin{equation}\label{E2}
U(2)_j=\left\{\begin{array}{ll}
\frac{U_{j+1}}{x^{p^{2j-2}}}&\mbox{ if $j$ is odd}\\
\frac{U_{j+1}}{x^{p^{2j-1}}}&\mbox{ if $j$ is even}
\end{array}
\right.
\end{equation}

and for $k\ge 3$, the generating sequence $\{U(k)_j\}$ in $A_k$ is defined by 
\begin{equation}\label{B1}
U(k)_0=U(k-1)_1=
\left\{\begin{array}{ll}
\frac{U_{k-1}}{x^{p^{2k-6}}U_{k-3}}&\mbox{ if $k$ is odd}\\
\frac{U_{k-1}}{x^{p^{2k-5}}U_{k-3}}&\mbox{ if $k$ is even}
\end{array}\right.
\end{equation}

and for $k$ odd, $j\ge 1$,

\begin{equation}\label{B2}
U(k)_j=\left\{\begin{array}{ll}
\frac{U_{j+k-1}}{x^{p^{2(j+k)-5}}U_{k-2}^{p^{2j-2}}}&\mbox{ if $j$ is odd}\\
\frac{U_{j+k-1}}{x^{p^{2(j+k)-6}}U_{k-2}^{p^{2j-3}}}&\mbox{ if $j$ is even}\\
\end{array}
\right.
\end{equation}

and for $k$ even, $j\ge 1$,

\begin{equation}\label{B3}
U(k)_j=\left\{\begin{array}{ll}
\frac{U_{j+k-1}}{x^{p^{2(j+k)-6}}U_{k-2}^{p^{2j-2}}}&\mbox{ if $j$ is odd}\\
\frac{U_{j+k-1}}{x^{p^{2(j+k)-5}}U_{k-2}^{p^{2j-1}}}&\mbox{ if $j$ is even}\\
\end{array}
\right.
\end{equation}

Further, there exist units $\delta(k)_j$ in $A_k$ with $\delta(k)_j\equiv 1\mbox{ mod }m_{A_k}$ such that if $k$ is odd, then
\begin{equation}\label{A1}
U(k)_2=U(k)_1^p-\delta(k)_0U(k)_0
\end{equation}
and for $j\ge 2$,
\begin{equation}\label{A2}
U(k)_{j+1}=\left\{\begin{array}{ll}
U(k)_j^p-\delta(k)_{j+1}U(k)_0^{p^{2j-2}}U(k)_{j-1}&\mbox{ if $j$ is odd}\\
U(k)_j^{p^3}-\delta(k)_{j+1}U(k)_0^{p^{2j-1}}U(k)_{j-1}&\mbox{ if $j$ is even}
\end{array}\right.
\end{equation}

and if $k$ is even, then 

\begin{equation}\label{A3}
U(k)_2=U(k)_1^{p^3}-\delta(k)_0U(k)_0
\end{equation}
and for $j\ge 2$,
\begin{equation}\label{A4}
U(k)_{j+1}=\left\{\begin{array}{ll}
U(k)_j^{p^3}-\delta(k)_{j+1}U(k)_0^{p^{2j-2}}U(k)_{j-1}&\mbox{ if $j$ is odd}\\
U(k)_j^{p}-\delta(k)_{j+1}U(k)_0^{p^{2j-3}}U(k)_{j-1}&\mbox{ if $j$ is even}
\end{array}\right.
\end{equation}
\end{Proposition}

\begin{proof} The fact that the sequence $\{U_j\}$ is a generating sequence of a unique $k$-valuation $\nu_1$ of $K_1$ was shown before Proposition \ref{Prop87}.

The remainder of the proposition is proved by induction on $k$. By (\ref{eq61}) and (\ref{eq60}) if $k=1$ and (\ref{A1}) - (\ref{A4}) if $k>1$
and by Theorem 7.1 \cite{CV1}, there exists a generating sequence $\{\tilde U(k+1)\}$ in $A_{k+1}$ defined if $k$ is odd by 
\begin{equation}\label{eq71}
\tilde U(k+1)_0=U(k)_1, \tilde U(k+1)_1=\frac{U(k)_2}{U(k)_1^p}
\end{equation}
and for $j\ge 2$,
\begin{equation}\label{eq72}
\tilde U(k+1)_j=\left\{\begin{array}{ll}
\frac{U(k)_{j+1}}{U(k)_1^{p^{2j-1}}}\mbox{ if $j$ is odd}\\
\frac{U(k)_{j+1}}{U(k)_1^{p^{2j}}}\mbox{ if $j$ is even}
\end{array}\right.
\end{equation}
and if $k$ is even, then

\begin{equation}\label{eq73}
\tilde U(k+1)_0=U(k)_1, \tilde U(k+1)_1=\frac{U(k)_2}{U(k)_1^{p^3}}
\end{equation}
and for $j\ge 2$,
\begin{equation}\label{eq74}
\tilde U(k+1)_j=\left\{\begin{array}{ll}
\frac{U(k)_{j+1}}{U(k)_1^{p^{2j+1}}}\mbox{ if $j$ is odd}\\
\frac{U(k)_{j+1}}{U(k)_1^{p^{2j}}}\mbox{ if $j$ is even}
\end{array}\right.
\end{equation}

To prove that the conclusions of the proposition hold for $\{U(k+1)_i\}$, we use the induction assumption and the appropriate equations (\ref{eq65}), (\ref{C1}) or (\ref{C2}) which are stated below to first verify that there are units $\lambda(k+1)_j$ in $A_{k+1}$
with $\lambda(k+1)_j\equiv 1\mbox{ mod }m_{A_{k+1}}$ such that 
\begin{equation}\label{eq75}
U(k+1)_j=\lambda(k+1)_j\tilde U(k+1)_j
\end{equation}
for all $j$, verifying that $\{U(k+1)_j\}$ is a generating sequence in $A_{k+1}$ for $\nu_1$, and   then 
that the appropriate equations (\ref{A1}) - (\ref{A4})  hold and finally that the appropriate equations (\ref{E1}) and (\ref{E2}),
or (\ref{B1}) - (\ref{B3}) hold.

We now state and prove the equations (\ref{eq65}), (\ref{C1}) and (\ref{C2}).

There exists a unit $\epsilon(2)$ in $A_2$ with $\epsilon(2)\equiv 1\mbox{ mod }m_{A_2}$ such that
\begin{equation}\label{eq65}
U_1^p=\epsilon(2) x.
\end{equation}

Equation (\ref{eq65}) follows from (\ref{eq61}) and the fact (Theorem 7.1 \cite{CV1}) that $A_2$ has regular parameters $\overline x_{A_2}$ and $\overline v_{A_2}$ defined by
\begin{equation}\label{eq76}
U(1)_0=x=\overline x_{A_2}^p(\overline v_{A_2}+1), U(1)_1=v=\overline x_{A_2}
\end{equation}

There exists a unit $\epsilon(k+1)\in A_{k+1}$ with $\epsilon(k+1)\equiv 1\mbox{ mod }m_{A_{k+1}}$ such that
\begin{equation}\label{C1}
U(k)_1^p=\epsilon(k+1)\frac{U_{k-1}}{U_{k-2}^p}
\end{equation}
if  $k\ge 2$ is odd and
\begin{equation}\label{C2}
U(k)_1^{p^3}=\epsilon(k+1)\frac{U_{k-1}}{U_{k-2}^{p^3}}
\end{equation}
if  $k\ge 2$ is even.

We now simultaneously verify the equations (\ref{C1}) and (\ref{C2}). 
 We first verify  that $A_{k+1}$ has regular parameters $\overline x_{A_{k+1}}$ and $\overline v_{A_{k+1}}$ defined if $k$ is odd by
\begin{equation}\label{eq77}
U(k)_0=\overline x_{A_{k+1}}^p(\overline v_{A_{k+1}}+1), U(k)_1=\overline x_{A_{k+1}}
\end{equation}
and if $k$ is even by 
\begin{equation}\label{eq78*}
U(k)_0=\overline x_{A_{k+1}}^{p^3}(\overline v_{A_{k+1}}+1), U(k)_1=\overline x_{A_{k+1}}.
\end{equation}
If $k$ is odd,  we have that $p\nu_1(U(k)_1)=\nu_1(U(k)_0)$ by (\ref{A1}) and if $k$ is even we have that  $p^3\nu_1(U(k)_1)=\nu_1(U(k)_0)$ by (\ref{A3}),  so (\ref{eq77}) and (\ref{eq78*}) follow from  (\ref{A1}), (\ref{A3}) and Theorem 7.1 \cite{CV1}. Thus for $j\le k$, $U_j$ is a power of $U(k+1)_0$ times a unit in $A_{k+1}$ which is equivalent to $1 \mbox{ mod } m_{A_{k+1}}$. Thus by (\ref{eq60}), there exists a unit $\epsilon(k+1)$ in $A_{k+1}$ with $\epsilon(k+1)\equiv 1 \mbox{ mod }m_{A_{k+1}}$ such that 
$$
U_k^p=\epsilon(k+1)x^{p^{2k-2}}U_{k-1}\mbox{ if $k$ is odd}
$$
 and 
$$
U_k^{p^3}=\epsilon(k+1)x^{p^{2k-1}}U_{k-1}\mbox{ if $k$ is even.}
$$
Thus if $k$ is odd,
$$
\begin{array}{lll}
U(k)_1^p&=& \left(\frac{U_k}{x^{p^{2k-3}}U_{k-2}}\right)^p=\frac{U_k^p}{x^{p^{2k-2}}U_{k-2}^p}\\
&=& \frac{\epsilon(k+1)x^{p^{2k-2}}U_{k-1}}{x^{p^{2k-2}}U_{k-2}^p}
=\epsilon(k+1)\frac{U_{k-1}}{U_{k-2}^p}.
\end{array}
$$ 
if $k$ is even,
$$
\begin{array}{lll}
U(k)_1^{p^3}&=& \left(\frac{U_k}{x^{p^{2k-4}}U_{k-2}}\right)^{p^3}=\frac{U_k^{p^3}}{x^{p^{2k-1}}U_{k-2}^{p^3}}\\
&=& \frac{\epsilon(k+1)x^{p^{2k-1}}U_{k-1}}{x^{p^{2k-1}}U_{k-2}^{p^3}}
=\epsilon(k+1)\frac{U_{k-1}}{U_{k-2}^{p^3}}.
\end{array}
$$ 
\end{proof}

\begin{Lemma}\label{Lemma64}
We have that
\begin{equation}\label{eq62'}
U_0=x=Q_0, U_1=Q_1^p-x^cy,
\end{equation}
and for $j\ge 1$, if $j$ is odd, then
\begin{equation}\label{eq62}
U_{j+1}=Q_{j+1}+x^{p^{2j-2}(\sum_{j'=0}^{\frac{j-1}{2}}\frac{1}{p^{4j'}})}f_{j+1}(x,y)
\end{equation}
where $f_{j+1}\in k[[x]][y]$, $x$ divides $f_{j+1}$ in $k[[x]][y]$, and $\mbox{deg}_yf_{j+1}=p^{2j-1}$.

If $j$ is even, then
\begin{equation}\label{eq63}
U_{j+1}=Q_{j+1}^p+x^{p^{2j-1}(\sum_{j'=0}^{\frac{j}{2}-1}\frac{1}{p^{4j'}})}f_{j+1}(x,y)
\end{equation}
where $f_{j+1}\in k[[x]][y]$, $x$ divides $f_{j+1}$ in $k[[x]][y]$, and $\mbox{deg}_yf_{j+1}=p^{2j}$.
\end{Lemma} 

\begin{proof} We have that
$$
U_1=v=y^p-x^cy=Q_1^p-x^cy,
$$
$$
U_2=v^p-x=Q_2-x^{cp}y^p,
$$
verifying (\ref{eq62}) for $j=1$ and
$$
U_3=U_2^{p^3}-x^{p^3}U_1=Q_3^p-x^{cp^4}y^p+x^{p^3+c}y
$$
verifying (\ref{eq63}) for $j=2$.

We prove the equations (\ref{eq62}) and (\ref{eq63})  for $j\ge 3$ by induction. First assume that $j$ is odd. We have that
$$
\begin{array}{lll}
U_{j+1}&=& U_j^p-x^{p^{2j-2}}U_{j-1}\\
&=& Q_{j+1}+x^{p^{2j-2}(\sum_{j'=0}^{\frac{j-3}{2}}\frac{1}{p^{4j'}})}f_j^p
-x^{p^{2j-2}+p^{2j-6}(\sum_{j'=0}^{\frac{j-3}{2}}\frac{1}{p^{4j'}})}f_{j-1}.
\end{array}
$$
The formula (\ref{eq62}) then follows since
$$
p^{2j-2}(\sum_{j'=0}^{\frac{j-1}{2}}\frac{1}{p^{4j'}})=p^{2j-2}(\sum_{j'=0}^{\frac{j-3}{2}}\frac{1}{p^{4j'}})+1
$$
and
$$
p^{2j-2}+p^{2j-6}(\sum_{j'=0}^{\frac{j-3}{2}}\frac{1}{p^{4j'}})=p^{2j-2}(\sum_{j'=0}^{\frac{j-1}{2}}\frac{1}{p^{4j'}}).
$$

Now assume that $j\ge 3$ is even. We have that
$$
\begin{array}{lll}
U_{j+1}&=& U_j^{p^3}-x^{p^{2j-1}}U_{j-1}\\
&=& Q_{j+1}^p+x^{p^{2j-1}(\sum_{j'=0}^{\frac{j}{2}-1}\frac{1}{p^{4j'}})}f_j^{p^3}
-x^{p^{2j-1}+p^{2j-5}(\sum_{j'=0}^{\frac{j}{2}-2}\frac{1}{p^{4j'}})}f_{j-1}.
\end{array}
$$
The formula (\ref{eq63}) then follows since
$$
p^{2j-1}+p^{2j-5}(\sum_{j'=0}^{\frac{j}{2}-2}\frac{1}{p^{4j'}})=p^{2j-1}(\sum_{j'=0}^{\frac{j}{2}-1}\frac{1}{p^{4j'}}).
$$
\end{proof}

\begin{Proposition}\label{Prop90}
We have $V_{\nu_1}=V_{\nu^*}\cap K_1$ and $\nu^*|K_1=\nu_1$. Further, $\nu^*$ is the unique extension of $\nu_1$ to $K^*$.
\end{Proposition}

\begin{proof} Since $\nu_1(x)=\nu^*(x)=1$, it suffices to  show that $V_{\nu_1}\subset V_{\nu^*}$. From (\ref{eq78}) and the fact that
$V_{\nu_1}=\cup A_j$, we need only show that $v_{A_{j+1}}\in V_{\nu^*}$ for all $j$. From equation (72) of \cite{CP}, we obtain
that
\begin{equation}\label{eq101}
\nu^*(Q_{j+1})=\overline \beta_{j+1}<p^{2j-2}(\sum_{j'=0}^{\frac{j-1}{2}}\frac{1}{p^{4j'}})+1
\end{equation}
if $j$ is odd, and
\begin{equation}\label{eq102}
\nu^*(Q_{j+1}^p)=p\overline \beta_{j+1}<p^{2j-1}(\sum_{j'=0}^{\frac{j}{2}-1}\frac{1}{p^{4j'}})+1
\end{equation}
if $j$ is even. Thus by Lemma \ref{Lemma64},
\begin{equation}\label{eq91*}
\nu^*(U_{j+1})=\left\{\begin{array}{ll}
\nu^*(Q_{j+1})&\mbox{ if $j$ is odd}\\
p\nu^*(Q_{j+1})&\mbox{ if $j$ is even}
\end{array}\right.
\end{equation}
Thus by (\ref{eq78}) and Proposition 7.40 \cite{CP},
$$
\nu^*(v_{A_{j+1}})=\left\{\begin{array}{ll}
\nu^*(y_{S_{j+1}})&\mbox{ if $j$ odd}\\
p\nu^*(y_{S_{j+1}})&\mbox{ if $j$ even}
\end{array}\right.
$$
The fact that $\nu^*$ is the unique extension of $\nu$ to $K^*$ follows since $\nu_1|K=\nu^*|K=\nu$ by Proposition 7.40 \cite{CP}, and $\nu^*$ is the unique extension of $\nu$ to $K^*$ by (1) of Theorem 7.38 \cite{CP}.
\end{proof}

\begin{Theorem}\label{Theorem71} For all $j$, $S_j$ is a finite extension of $A_j$. If $j$ is odd, we have expressions
\begin{equation}\label{eq80}
x_{A_{j+1}}=\tau_{j+1}x_{S_{j+1}}^p, v_{A_{j+1}}=\gamma_{j+1}y_{s_{j+1}}+x_{S_{j+1}}\Omega_{j+1}
\end{equation}
where $\tau_{j+1},\gamma_{j+1}$ are units in $S_{j+1}$ and $\Omega_{j+1}\in S_{j+1}$.

If $j$ is even, we have expressions
\begin{equation}\label{eq81}
x_{A_{j+1}}=\tau_{j+1}x_{S_{j+1}}, v_{A_{j+1}}=\gamma_{j+1}y_{S_{j+1}}^p+x_{S_{j+1}}\Omega_{j+1}
\end{equation}
where $\tau_{j+1},\gamma_{j+1}$ are units in $S_{j+1}$ and $\Omega_{j+1}\in S_{j+1}$.
\end{Theorem}

\begin{proof} 
For $i\le j$, $Q_{i}$ are units in $S_{j+1}$ times a power of $x_{S_{j+1}}$. By (\ref{eq62}), (\ref{eq63}) and (\ref{eq101}) and (\ref{eq102}), we then have that
$U_{i}$ are units in $S_{j+1}$ times a power of $x_{S_{j+1}}$ and for $i\le j$,
\begin{equation}\label{eq92*}
U_{i}=\left\{\begin{array}{ll}
\alpha Q_{i}^p&\mbox{ where $\alpha\in S_{j+1}$ is a unit if $i$ is odd}\\
\alpha  Q_{i}&\mbox{ where $\alpha\in S_{j+1}$ is a unit if $i$ is even.}
\end{array}\right.
\end{equation}
We have from (\ref{eq101}) that
\begin{equation}\label{eq93*}
\nu^*(Q_{j-1})=\overline\beta_{j-1}=p^{2j-6}(\sum_{j'=0}^{j-2}\frac{1}{p^{4j'}})<
p^{2j-6}(\sum_{j'=0}^{\frac{j-1}{2}-1}\frac{1}{p^{4j'}})+1=
p^{2j-2}(\sum_{j'=0}^{\frac{j-1}{2}}\frac{1}{p^{4j'}})+1-p^{2j-2}
\end{equation}
if $j$ is odd and by (\ref{eq102}),
\begin{equation}\label{eq94*}
p\nu^*(Q_{j-1})=p\overline\beta_{j-1}=p^{2j-5}(\sum_{j'=0}^{j-2}\frac{1}{p^{4j'}})<
p^{2j-5}(\sum_{j'=0}^{\frac{j}{2}-2}\frac{1}{p^{4j'}})+1=
p^{2j-1}(\sum_{j'=0}^{\frac{j}{2}-1}\frac{1}{p^{4j'}})+1-p^{2j-1}
\end{equation}
if $j$ is even.

If $j$ is odd, by equations (\ref{eq78}), (\ref{eq62}), (\ref{eq92*}) and (\ref{eq93*}) and Proposition \ref{Prop87}, we have an expression
$$
v_{A_{j+1}}=\frac{U_{j+1}}{x^{p^{2j-2}}U_{j-1}}=\frac{Q_{j+1}+x^{p^{2j-2}(\sum_{j'=0}^{\frac{j-1}{2}}\frac{1}{p^{4j'}})}f_{j+1}}
{\alpha x^{p^{2j-2}}Q_{j-1}}=\gamma y_{S_{j+1}}+x_{S_{j+1}}\Omega
$$
for some unit $\gamma\in S_{j+1}$ and $\Omega\in S_{j+1}$.

If $j$ is even, by equations (\ref{eq78}), (\ref{eq63}), (\ref{eq92*}) and (\ref{eq94*}) and Proposition \ref{Prop87}, we have an expression
$$
v_{A_{j+1}}=\frac{U_{j+1}}{x^{p^{2j-1}}U_{j-1}}=\frac{Q_{j+1}^p+x^{p^{2j-1}(\sum_{j'=0}^{\frac{j}{2}-1}\frac{1}{p^{4j'}})}f_{j+1}}
{\alpha x^{p^{2j-1}}Q_{j-1}^p}=\gamma y_{S_{j+1}}^p+x_{S_{j+1}}\Omega
$$
for some unit $\gamma\in S_{j+1}$ and $\Omega\in S_{j+1}$.

By (\ref{eq92*}), we have that
$$
x_{A_{j+1}}=\frac{U_j}{x^{p^{2j-3}}U_{j-2}}=\frac{\alpha Q_j^p}{x^{p^{2j-3}}Q_{j-2}^p}=\alpha x_{S_{j+1}}^p
$$
for some unit $\alpha\in S_{j+1}$ if $j$ is odd and
$$
x_{A_{j+1}}=\frac{U_j}{x^{p^{2j-4}}U_{j-2}}=\frac{\alpha Q_j}{x^{p^{2j-4}}Q_{j-2}}=\alpha x_{S_{j+1}}
$$
for some unit $\alpha\in S_{j+1}$ if $j$ is even.

The extension $A_j\rightarrow S_j$ is finite for all $j$ since each $A_j\rightarrow S_j$ is quasi finite and $\nu^*$ is the unique extension of
$\nu$ to $K^*$.

\end{proof}

\begin{Theorem}\label{Theorem72} For all $j$, $A_j$ is a finite extension of $R_j$. If $j$ is odd, we have expressions
\begin{equation}\label{eq83}
u_{R_{j+1}}=\sigma_{j+1}x_{A_{j+1}}, v_{R_{j+1}}=\lambda_{j+1}v_{A_{j+1}}^p+x_{A_{j+1}}\Lambda_{j+1}
\end{equation}
where $\sigma_{j+1},\lambda_{j+1}$ are units in $A_{j+1}$ and $\Lambda_{j+1}\in A_{j+1}$.

If $j$ is even, we have expressions
\begin{equation}\label{eq84}
u_{R_{j+1}}=\sigma_{j+1}x_{A_{j+1}}^p, v_{R_{j+1}}=\lambda_{j+1}v_{A_{j+1}}+x_{A_{j+1}}\Lambda_{j+1}
\end{equation}
where $\sigma_{j+1},\lambda_{j+1}$ are units in $A_{j+1}$ and $\Lambda_{j+1}\in A_{j+1}$.
\end{Theorem}

\begin{proof} We have that $R_{j+1}\subset A_{j+1}$ and $R_{j+1}\rightarrow A_{j+1}$ is finite since $R_{j+1}\rightarrow S_{j+1}$ and $A_{j+1}\rightarrow S_{j+1}$ are finite, $A_{j+1}$ is normal and $\mbox{QF}(R_{j+1})\subset \mbox{QF}(A_{j+1})$. The expressions (\ref{eq83}) and (\ref{eq84}) follow from (2) of Theorem 7.38 \cite{CP} and Theorem \ref{Theorem71}.
\end{proof}

\end{document}